\documentclass[11pt]{amsart}

\usepackage{amssymb}
\usepackage{amsmath, amsfonts,enumerate}
\usepackage{hyperref}
\usepackage{graphics}
\usepackage{amsthm}
\usepackage{latexsym,bm}
\usepackage{euscript}

\let\cal\mathcal

\makeatletter \@mparswitchfalse \makeatother
\newtheorem{theorem}{Theorem}
\newtheorem{lemma}[theorem]{Lemma}

\newtheorem{corollary}[theorem]{Corollary}
\newtheorem{proposition}[theorem]{Proposition}
\newtheorem{remark}[theorem]{Remark}

\theoremstyle{definition}
\newtheorem{definition}[theorem]{Definition}

\theoremstyle{remark}

\numberwithin{equation}{section}
\numberwithin{theorem}{section}

\def\M{\cal{M}}
\def\H{\cal{H}}
\let\<\langle
\def\ch{\raise 0.5ex \hbox{$\chi$}}
\def\T{\tau}
\def\E{\cal{E}}

\let\\\cr

\let\epsilon\varepsilon

\def\log{\operatorname{log}}
\renewcommand{\a}{\alpha}
\renewcommand{\b}{\beta}
\newcommand{\g}{\gamma}

\newcommand{\s}{\sigma}

\newcommand{\N}{\cal{N}}
\newcommand{\h}{\mathsf{h}}

\newcommand{\Tr}{\mbox{\rm Tr}}
\newcommand{\tr}{\mbox{\rm tr}}

\newcommand{\bmo}{\mathsf{bmo}}
\newcommand{\BMO}{{\mathcal {BMO}}}

\begin{document}

\title[Martingale Hardy spaces]{P. ~Jones'  interpolation theorem  for   noncommutative martingale Hardy spaces}

\author[Randrianantoanina]{Narcisse Randrianantoanina}
\address{Department of Mathematics, Miami University, Oxford,
Ohio 45056, USA}
 \email{randrin@miamioh.edu}
 


\begin{abstract}
Let $\M$ be a semifinite von Nemann algebra equipped with an increasing filtration $(\M_n)_{n\geq 1}$  of (semifinite) von Neumann subalgebras  of $\M$.  For
  $0<p \leq\infty$,  let  $\h_p^c(\M)$ denote the  noncommutative  column    conditioned martingale Hardy space  associated with    the  filtration $(\M_n)_{n\geq 1}$ and the index $p$.
We prove that   for $0<p<\infty$, the  compatible couple
$\big(\h_p^c(\M), \h_\infty^c(\M)\big)$ is $K$-closed in the  couple  $\big(L_p(\N), L_\infty(\N) \big)$ for an appropriate amplified semifinite von Neumann algebra $\N \supset \M$. This may be viewed as a noncommutative analogue of P. Jones interpolation of the couple $(H_1, H_\infty)$.

As an application, we prove  a general  automatic transfer of  real  interpolation results from couples of symmetric quasi-Banach function spaces to the corresponding couples of noncommutative conditioned martingale Hardy spaces. More precisely,  assume that  $E$  is    a symmetric quasi-Banach function space on $(0, \infty)$ satisfying some natural conditions, $0<\theta<1$,  and $0<r\leq \infty$. If $(E,L_\infty)_{\theta,r}=F$, then 
\[
\big(\h_E^c(\M), \h_\infty^c(\M)\big)_{\theta, r}=\h_{F}^c(\M).
\]
As an  illustration,   we  obtain     that if $\Phi$ is an Orlicz  function that is $p$-convex and $q$-concave for some  $0<p\leq q<\infty$, then   the following interpolation on the noncommutative column  Orlicz-Hardy space holds: for $0<\theta<1$, $0<r\leq \infty$, and  $\Phi_0^{-1}(t)=[\Phi^{-1}(t)]^{1-\theta}$ for $t>0$, 
  \[
  \big(\h_\Phi^c(\M), \h_\infty^c(\M)\big)_{\theta, r}=\h_{\Phi_0, r}^c(\M)
  \]
  where $\h_{\Phi_0,r}^c(\M)$ is the noncommutative column Hardy space associated with the Orlicz-Lorentz space $L_{\Phi_0,r}$.
\end{abstract}

\maketitle

\section{Introduction}

The theory of noncommutatiive martingales  is now well-established as a useful tool in  various aspects of noncommutative analysis,  quantum probability theory, and operator algebras. Various types of Hardy spaces arising from noncommutative martingales have played significant role in the development of the theory for the last few decades. We recall that the noncommutative Burkholder-Gundy inequalities  proved by Pisier and Xu in \cite{PX} that triggered  the modern phase of the study of noncommutative martingales  were  based on  consideration of  column/row  Hardy spaces and how they should be mixed according  the indices.   Likewise, the noncommutative Burkholder inequalities due to Junge and Xu in \cite{JX} can be formulated using the conditioned versions of the ones from \cite{PX} together with 
another type called  diagonal Hardy spaces. We should  also mention  that the  theory of maximal inequalities for noncommutative martingales developed by Junge in \cite{Ju} generates  another type of Hardy spaces in this context.

Interpolations between  Hardy spaces  from various fields of mathematics  have a long history and proven to be very useful in  many areas such as harmonic analysis, PDE's, Banach space  theory,  among others. Although some  of these Hardy spaces are closely related, we will only   focus  on those  from martingale theory.
We refer to \cite{Jason-Jones,Weisz, Xu-inter2, Xu-inter3} for  background concerning  interpolations  of Hardy spaces  from classical martingale theory that are relevant for our purpose.

The primary objective  of the present  paper  is to further advance the topic of   interpolations between noncommutative column/row  conditioned martingale Hardy spaces from noncommutative martingale theory. 
We refer to the body of the paper for unexplained notation below. 
We recall that the study of interpolations of noncommutative martingale Hardy spaces was  initiated by Musat in \cite{Musat-inter} (see also \cite{Junge-Musat})
where the complex interpolation of the compatible  couple $(\H_1, \BMO)$ was given. Later, Bekjan  {\it et al.} established in \cite{Bekjan-Chen-Perrin-Y} that  the analogue of  Musat's result is valid for  couple of  column/row conditioned Hardy spaces. More precisely,  they   obtained  the corresponding result for the compatible couple $(\h_1^c, \bmo^c)$. More recently,  the case of quasi-Banach space couple $(\h_p^c, \bmo^c)$  for $0<p<\infty$ were obtained in  \cite{Ran-Int}  using the real interpolation method. To the best of our knowledge, the articles \cite{Bekjan-Chen-Perrin-Y, Junge-Musat, Musat-inter, Ran-Int} are the only available references in the literature that contain essential   progress in understanding  interpolation spaces  of  noncommutative martingale Hardy spaces to date.
A common theme in  the above cited  articles   is that  all couples considered have one of the endpoints consisting of  appropriate types of noncommutative  martingale $BMO$ spaces.  Naturally, the noncommutative Fefferman-Stein   duality of $H_1$ and $BMO$  established in \cite{JX, PX} in the setting of  noncommutative martingale spaces  plays prominent role in achieving the  appropriate  martingale BMO-space as endpoints of  the  various interpolation
couples   considered thus far. 

In the present paper, we deviate from  the papers cited above  and consider a previously   untouched topic  which   we 
may view as   a  version of the Peter Jones's interpolation  for  the couple  of classical  Hardy spaces $(H_1, H_\infty)$ in the setting of 
Hardy spaces associated with noncommutative martingales.  We refer to  \cite{BENSHA, Jones1, Jones2, Pisier-int} for background concerning Peter Jones's result.

To better explain our consideration, we make the notation a little more precise. Assume that  $\M$  is   a semifinite von Neumann algebra equipped with a faithful normal semifinite trace $\T$. For $0<p\leq \infty$,  $\h_p^c(\M)$ denotes the noncommutative column conditioned martingale Hardy space associated to the index $p$ and fixed  filtration  $(\M_n)_{n\geq 1}$ for which we refer to the preliminary section below for detailed description. These spaces were heavily used in  the articles \cite{Ju, JX}. The exact formulation of our primary  result  takes into account 
 a   highly nontrivial   fact  proved  by Junge in \cite{Ju} that   there exists a  semifinite von Neumann algebra $\N$ with $\M \subset \N$ and such that   for every $0<p \leq \infty$,  the (quasi) Banach space $\h_p^c(\M)$ isometrically embeds into the noncommutative space $L_p(\N)$. Thus,  it  may be viewed as a  subspace of $L_p(\N)$.  The main theorem  in the paper states  that for $0<p<\infty$, the compatible couple $(\h_p^c(\M),\h_\infty^c(\M))$ is $K$-closed in the couple $(L_p(\N), L_\infty(\N))$ in the sense of Pisier (\cite{Pisier-int}).  We refer to Theorem~\ref{main-closed}  below for precise formulation. As an immediate consequence, we deduce that the family $\{\h_p^c(\M)\}_{0<p \leq \infty}$ forms a real interpolation scale  in the following sense: if $0<p<\infty$, $0<\theta<1$, and 
 $1/q=(1-\theta)/p$, then
 \[
 \big(\h_p^c(\M), \h_\infty^c(\M) \big)_{\theta,q}=\h_q^c(\M)
 \]
where $(\cdot, \cdot)_{\theta,q}$ denotes the real interpolation method. In fact, our  $K$-closed result is  quite  flexible that  when combined  with some known  general results about  $K$-functionals, it allows  us to prove   far more general interpolation   results. For instance, we obtain  an  automatic transfer of  interpolation of couples of symmetric function spaces to the corresponding couples of  noncommutative conditioned column Hardy spaces. More precisely, if $E$ is a  quasi-Banach function space that is an interpolation space of the couple $(L_p,L_q)$ for some $0<p<q<\infty$ and $F$ is an interpolation of the couple $(E, L_\infty)$, then 
$\h_F^c(\M)$  is an interpolation space of the couple $(\h_E^c(\M), \h_\infty^c(\M))$.
These type of results appeared to be new even for the classical  martingale settings. We refer to Corollary~\ref{lifting} for  precise formulation.

Our method of proof  combines techniques from interpolation theory and martingale theory. We should  note that duality arguments were used in the papers
\cite{Bekjan-Chen-Perrin-Y, Junge-Musat, Musat-inter}
  in order to 
reach $BMO$-spaces as one of the endpoints. This is no longer available when working with $\h_\infty^c(\M)$. Our approach is based on  explicit decompositions for martingales from $\h_2^c(\M) +\h_\infty^c(\M)$ in order to get specific formulae for $K$-functionals for the couple $(\h_2^c(\M),\h_\infty^c(\M))$.  As expected, the so-called Cuculescu's  projections  are instrumental in constructing such decompositions. Once we understand the couple $(\h_2^c(\M), \h_\infty^c(\M))$, we combine it with some previously known results from \cite{Ran-Int} via a Wolff type  theorem for $K$-closed couples due to Kislyakov and Xu  to  deduce the general statement for the couple $(\h_p^c(\M), \h_\infty^c(\M))$ for any given $0<p<2$. 
The generalization to the couple $(\h_E^c(\M), \h_\infty^c(\M))$ where $E$ is a  symmetric quasi-Banach function space is achieved by   a reiteration formula for interpolation  with function space parameters.

\medskip

 Recently,  there have  been  renewed interests on martingale Hardy spaces associated with Orlicz  function spaces in the classical setting. We refer to  recent articles  \cite{L-T-Zhou, L-Weisz-Xie} for  more perseptive  and background on real interpolations of couples involving  martingale Orlicz-Hardy spaces in the classical setting. Motivated by these new developments, we show as an illustration of our unified approach of lifting interpolation results from couples of  symmetric function spaces  to the corresponding 
noncommutative column Hardy spaces, that  results from \cite{L-T-Zhou, L-Weisz-Xie} admit noncommutative counterparts. These include  real interpolations of the couples $(\h_\Phi^c(\M),\h_\infty^c(\M))$ and $(\h_\Phi^c(\M),\bmo^c(\M))$ for  Orlicz function $\Phi$ satisfying some natural conditions. 

\medskip

The paper is organized as follows. In Section~2, we give a brief introduction  of noncommutative  spaces and  review the constructions of various Hardy spaces associated to symmetric  function spaces. This  section also includes discussions on  some concepts from interpolation  theory that we need in order to precisely state our main theorem.
Section~3 is where we provide the formulation  and proof  of our result in the form of $K$-closed couples together with some extensions and consequences. The section also contains two separate paragraphs dealing with the specific examples of noncommutative conditioned  Hardy spaces associated with Orlicz function spaces and generalized Lorentz spaces respectively. 

In the last part of the paper, we include an appendix section  where we discuss  some applications  of the  new development made in earlier sections  to martingale inequalities.   We provide  improvement on  all results from \cite[Section~4]{Ran-Int} for  inequalities involving martingales in symmetric spaces of measurable operators and corresponding moment inequalities associated with Orlicz functions.  We  single out  here  as one  of the improvements we made is  a resolution of  a problem left open in \cite{Ran-Wu-Xu} related to  Davis type decomposition for martingales in symmetric spaces of measurable operators. We show that the noncommutative Davis decomposition applies to any  martingale in 
the  Hardy space $\H_E^c(\M)$ if and only if $E$ is an interpolation  of  the couple $(L_1, L_2)$.

\section{Preliminaries}

In this section,  we collect some of the basic facts, notation and tools that  will be used in the paper.

\subsection{Noncommutative symmetric spaces}
Throughout,   $\M \subseteq  {\cal B}(\cal H)$   will denote  a semifinite von Neumann algebra  on some Hilbert space $\cal H$ (here, ${\cal B}(\cal H)$ is the algebra of all bounded operators on $\cal H$ with the usual operator norm).  It is assumed that $\M$ is equipped with a  fixed faithful normal semifinite trace $\T$. 
The identity in $\M$ will be denoted by ${\bf 1}$. Recall that a linear operator $x: {\cal D}(x) \to \cal{H}$, with domain $\cal{D}(x) \subseteq \cal{H}$ is said to be  \emph{affiliated with } $\M$ if $xu\subseteq ux$ for all unitary in the commutant $\M'$ of $\M$. A densely defined operator $x$ is affiliated with $\M$ if and only if  for every Borel set $B \subseteq \mathbb{R}$, $\ch_B(|x|) \in \M$  where  $\ch_B(|x|)$ is the spectral projection of $|x|$ associated  with $B$.
The closed and densely defined operator $x$, affiliated with $\M$, is called \emph{$\T$-measurable} if and only if there exists $s\geq 0$ such that $\T\big(\ch_{(s,\infty)}(|x|)\big )<\infty$. The collection of all $\T$-measurable operators is denoted by $\widetilde{\M}$. For $\epsilon, \delta>0$,  we set  $V(\epsilon,\delta)$ to be the set all 
$x\in \widetilde{\M}$  for which there exists a projection $p\in \M$ such that $p(\cal H) \subseteq \cal{D}(x)$,  $\|xp\|_\infty <\epsilon$,  and $\T({\bf 1}-p)<\delta$.
The system $\{V(\epsilon,\delta) : \epsilon,\delta>0\}$ forms a
base at $0$ for a metrizable Hausdorff topology in $\widetilde{\M}$, which is called the \emph{measure topology}. Equipped with the measure topology, $\widetilde{\M}$ is a complete topological $*$-algebra.  These facts can be found in \cite{N}.

  For $x \in \widetilde{\M}$, we recall   that its  \emph{generalized singular value } $\mu(x)$ 
is  the real-valued function defined by
\[
\mu_t(x) := \inf\big\{ s>0 : \T\big(\ch_{(s,\infty)}(|x|)\big)\leq t\big\}, \quad t>0.
\]
It follows directly from the definition that the function $\mu(x)$ is decreasing, right-continuous on the interval $[0,\infty)$ and $\lim_{t \to 0^+}\mu_t(x)=\|x\|_\infty$.
We refer to \cite{FK} for more information on generalized singular values.

We observe that
 if $\M$ is the abelian von Neumann algebra $L_\infty(0,\infty)$ with  the trace given by  integration  with respect to  the Lebesgue  measure, then $\widetilde{\M}$  becomes  the space of  all  measurable  complex functions  on $(0,\infty)$ which are bounded except on a set of finite measure and for  $f\in \widetilde{\M}$,  $\mu(f)$ is precisely  the usual decreasing rearrangement of  the function  $|f|$. We also note that in this case, convergence for measure topology coincides with  the usual notion of convergence in measure.

We will now review the construction of noncommutative symmetric spaces.
We denote by  $L_0$,  the space of  all measurable functions on the interval $(0,\infty)$. 
Recall that a  quasi-Banach function space  $(E,\|\cdot\|_E)$ of measurable functions  on the interval $(0,\infty)$ is called \emph{symmetric} if for any $g \in E$ and any $f \in L_0$ with $\mu(f) \leq \mu(g)$, we have $f \in E$ and $\|f\|_E \leq \|g\|_E$. 
Throughout, all function spaces are assumed to be defined on the interval $(0,\infty)$.

Let  $E$ be a symmetric quasi-Banach function space. We  define the  corresponding  noncommutative space by setting:
\begin{equation*}
E(\M, \T) = \Big\{ x \in
\widetilde{\M}\ : \ \mu(x) \in E \Big\}. 
\end{equation*}
Equipped with the  quasi-norm
$\|x\|_{E(\M,\T)} := \| \mu(x)\|_E$,   the linear space $E(\M,\T)$ becomes a complex quasi-Banach space (\cite{Kalton-Sukochev,X}) and is usually referred to as the \emph{noncommutative symmetric space} associated with $(\M,\T)$ corresponding to  $(E, \|\cdot\|_E)$. 
 We remark that  if $0< p<\infty$ and $E=L_p$, then $E(\M, \T)$ is exactly   the usual noncommutative $L_p$-space  $L_p(\M,\T)$ associated with  $(\M,\T)$. 
 In the sequel, $E(\M,\T)$ will be abbreviated to $E(\M)$. 
 
 Beside $L_p$-spaces, Lorentz spaces are also very involved  in   the subsequent part  of the paper.
Let $0<p,q\leq \infty$. The \emph{Lorentz space} $L_{p,q}$ is the space  of all $f \in L_0$ for which $\|f\|_{p,q}<\infty$ where
 \begin{equation*}
 \big\|f \big\|_{p,q}  =\begin{cases}
 \left(\displaystyle{\int_{0}^\infty \mu_{t}^{q}(f)\
d(t^{q/p})}\right)^{1/q},  &0< q < \infty; \\
\displaystyle{\sup_{t >0} t^{1/p} \mu_t(f)}, &q= \infty.
\end{cases} 
 \end{equation*}
If $1\leq q\leq p <\infty$ or $p=q=\infty$, then $L_{p,q}$ is a symmetric Banach function space.  If $1<p<\infty$ and $p\leq q\leq \infty$, then $L_{p,q}$ can be equivalently   renormed to become a symmetric Banach function (\cite[Theorem~4.6]{BENSHA}). In general, $L_{p,q}$ is only a symmetric quasi-Banach function space.

We conclude the subsection with a short  introduction of the  notion of  submajorization in  the sense of Hardy, Littlewood, and Polya.
If $x,y \in \widetilde{\M}$  then $x$ is said to be \emph{submajorized} by $y$  if for every $t>0$,  the inequality
\[
\int_0^t \mu_s(x)\ ds \leq \int_0^t \mu_s(y)\ ds 
\]
holds. In this case, we will write $x\prec\prec y$.

In the sequel, we will frequently use the submajorization  inequality
\begin{equation}\label{sub-sum}
\mu(x+y) \prec\prec \mu(x) +\mu(y).
\end{equation}
Another fact that is important below is that if $T: L_1(\M) +\M \longrightarrow  L_1(\M) +\M$ satisfies 
$\max\{ \| T: L_1(\M) \to L_1(\M)\|; \| T: \M \to \M\|\} \leq 1$ then  for every $x \in L_1(\M) +\M$, $Tx \prec\prec x$. This fact can be found in \cite[Proposition~4.1]{DDP3}. In particular, if $x \in L_1(\M)+\M$ and $(p_k)_{k\geq 1}$  is a sequence of  mutually disjoint projections from $\M$  then, 
\begin{equation}\label{sub-diagonal}
\sum_{k\geq 1} p_k x p_k \prec\prec x.
\end{equation}

  For more information on von Neumann algebras and  noncommutative integration, we refer to \cite{PX3, TAK,TAK2}.

  \subsection{Basic definitions and terminology from  interpolations} 
Let $(A_0, A_1)$ be a  compatible couple of quasi-Banach spaces  in the sense that both $A_0$ and $A_1$ embed continuously into some topological vector space $\mathcal{Z}$. This allows us to  define the spaces $A_0 \cap A_1$  and $A_0 +A_1$. These are quasi-Banach spaces when equipped with  quasi-norms:
\[
\big\|x  \big\|_{A_0 \cap A_1}=\max\Big\{ \big\|x  \big\|_{A_0 } , \big\|x  \big\|_{ A_1}\Big\}
\]
and
\[
\big\|x  \big\|_{A_0 + A_1}=\inf\Big\{ \big\|x_0  \big\|_{A_0 } + \big\|x_1  \big\|_{ A_1}: \, x=x_0 +x_1,\,  x_0 \in A_0,\,  x_1 \in A_1\Big\},
\]
respectively.  
\begin{definition}\label{def-interpolation}
A  quasi-Banach space $A$ is called  an \emph{interpolation space} for the couple $(A_0, A_1)$  if $A_0 \cap A_1 \subseteq A \subseteq A_0 +A_1$ and whenever  a bounded linear operator $T: A_0 +A_1\to A_0 +A_1$ is such that $T(A_0) \subseteq A_0$ and $T(A_1) \subseteq A_1$, we have $T(A)\subseteq A$
and 
\[
\big\|T:A\to A \big\|\leq c\max\left\{\big\|T :A_0 \to A_0\big\| , \big\|T: A_1 \to A_1\big\| \right\}
\] 
for some constant $c$. 
\end{definition}
If $A$ is an interpolation space for the couple $(A_0, A_1)$, we write $A\in {\rm Int}(A_0,A_1)$. Below, we are primarily  interested in an  interpolation method generally referred to as   the real  method.
  

We begin with   a short  discussion   of the  \emph{real interpolation} method.
 A fundamental notion  for the construction of   real interpolation spaces is the  \emph{$K$-functional} which we now describe. For $x \in A_0 +A_1$, we define the $K$-functional  by setting for $t>0$,
\[
K(x, t) =K\big(x,t; A_0,A_1\big)=\inf\Big\{ \big\|x_0  \big\|_{A_0 } + t\big\|x_1  \big\|_{ A_1}:\,  x=x_0 +x_1,\,  x_0 \in A_0,\,  x_1 \in A_1\Big\}.
\]
Note that for  each $t>0$, $x \mapsto K(x,t)$ gives an equivalent  quasi-norm on $A_0  +A_1$.

If $0<\theta<1$ and $0< \g<\infty$, we recall  that the real interpolation space $A_{\theta, \g}=(A_0, A_1)_{\theta, \g}$ by $x \in A_{\theta,\g}$ if and only if

\[
\big\| x \big\|_{(A_0,A_1)_{\theta, \g}} =\Big( \int_0^\infty \big(t^{-\theta}K\big(x, t; A_0,A_1\big)\big)^{\g }\ \frac{dt}{t} \Big)^{1/\g} <\infty.
\]
If $\g=\infty$, we define $ x \in A_{\theta,\infty}$ if and only if 
\[
\big\| x\big\|_{(A_0, A_1)_{\theta, \infty}}= \sup_{t>0} t^{-\theta} K(x, t; A_0, A_1)<\infty.
\]
For $0<\theta<1$ and $0<\g\leq \infty$, the functional $\|\cdot\|_{\theta,\g}$ is a  quasi-norm and $(A_{\theta,\g}, \|\cdot\|_{\theta,\g})$ is a quasi-Banach space. Moreover,
the space $A_{\theta,\g}$ is an interpolation space for the couple $(A_0, A_1)$  in the sense of Definition~\ref{def-interpolation}. There is also an equivalent description of $A_{\theta,\gamma}$  using the  dual notion of $J$-functionals but this will not be needed  for our purpose below.

\smallskip

It is worth mentioning   that the  real interpolation method is well understood for the couple $(L_{p_0}, L_{p_1})$ for both   the classical case and the  noncommutative case. We record here that
Lorentz spaces can be realized as  real interpolation spaces  for  the couple $(L_{p_0}, L_{p_1})$. More precisely, 
 if $\cal N$  is a semifinite von Neumann algebra, $0<p_0<p_1\leq \infty$, $0<\theta<1$, and $0<q\leq \infty$ then,   up to equivalent  quasi-norms (independent of $\cal N$),  
\[
\big(L_{p_0}(\cal N), L_{p_1}(\cal N)\big)_{\theta,q}= L_{p,q}(\cal N)
\]
where $1/p=(1-\theta)p_0 +\theta/p_1$. By reiteration,  if $0<\lambda, \g\leq \infty$, we have 
\begin{equation}\label{Lp}
\big(L_{p_0, \lambda}(\cal N), L_{p_1, \g}(\cal N)\big)_{\theta,q}= L_{p,q}(\cal N)
\end{equation}
 with equivalent  quasi-norms. These facts can be found in \cite{PX3} and will be used repeatedly throughout.

\medskip

A  more  general  real interpolation  type spaces will be essential  for our  consideration below. Recall that a quasi-Banach function space $\cal{F}$ has a monotone quasi-norm if  whenever $f,g \in  \cal{F}$, $|f| \leq |g| \implies \|f\|_{\cal{F}} \leq  \|g\|_{\cal{F}}$.
\begin{definition}
  An  interpolation space $E$  for a couple  of quasi-Banach spaces $(E_0,E_1)$ is said to be \emph{given by a $K$-method} if there exists a  quasi-Banach function space $\cal{F}$  with monotone quasi-norm such that $x \in E$ if and only if $t\mapsto K(x,t ; E_0,E_1) \in \cal{F}$ 
 and there exists a  constant $C_E>0$  such that
 \[
 C_E^{-1} \big\| t\mapsto K(x,t ; E_0,E_1)\big\|_{\cal{F}} \leq \big\|x\big\|_E \leq C_E \big\| t\mapsto K(x,t ; E_0,E_1)\big\|_{\cal{F}}.
 \]
In this case, we write $E=(E_0, E_1)_{\cal{F};K}$. 
\end{definition}

The following fact will be used in the sequel. 
 
 \begin{proposition}\label{K-method}
   Let  $0<p<q\leq \infty$. Every interpolation space   $E\in {\rm Int}(L_p,L_q)$ is given by a $K$-method.
  \end{proposition}
  
    For the Banach space range, this fact is known as a result of Brudnyi and Krugliak (see \cite[Theorem~6.3]{KaltonSMS}). An  argument for the   quasi-Banach space range can be found in     \cite{Ran-Int}.

\medskip

We now  review the primary topic of the paper.
The following concept on $K$-functionals was formally introduced by Pisier in \cite{Pisier-int} and will be essential in the subsequent discussions.
\begin{definition}
Let $(A_0, A_1)$ be a compatible couple of quasi-Banach spaces and $B_0$ (resp. $B_1$) be a subspace of $A_0$ (resp. $A_1$). The couple $(B_0,B_1)$ is said to be \emph{$K$-closed  in  the couple $(A_0,A_1)$ } if there exists a constant $C$ such that  for every $y \in B_0 +B_1$ and $t>0$,
\[
K(y,t ; B_0,B_1) \leq C K(y,t ; A_0,A_1).
\]
\end{definition}
Since the reverse inequality  is always valid (with constant $1$), $K$-closedness means that $K$-functionals of the couples $(B_0,B_1)$ and $(A_0, A_1)$ are equivalent  on $B_0 + B_1$ uniformly on $t>0$. Clearly, if $(B_0,B_1)$ is $K$-closed in $(A_0, A_1)$ then for every indices $0<\theta<1$ and  $0<\gamma\leq \infty$,
\[
(B_0, B_1)_{\theta, \gamma} = (B_0 + B_1) \cap (A_0, A_1)_{\theta, \gamma}.
\]
In fact,  this is the case for more general interpolation method based on $K$-functionals: if $\cal{F}$ is quasi-Banach function space with monotone quasi-norm, then 
\[
(B_0, B_1)_{\cal{F}; K} = (B_0 + B_1) \cap (A_0, A_1)_{\cal{F}; K}.
\]
That  is, one can deduce  interpolation results for the couple $(B_0,B_1)$ from the 
corresponding results on the larger couple. For more information and background on $K$-closed couples, we refer to \cite{Pisier-int, Ki, KiX}.

A Wolff-type result for $K$-functionals was proved in \cite{KiX}
(see also \cite{Ki}).  We state it here for further use.
\begin{theorem}[\cite{KiX}] \label{Wolff-K}
Let $(A_0,A_1)$ be a compatible couple of quasi-Banach spaces and $B_0$ (resp. $B_1$) be a subspace of $A_0$ (resp. $A_1$). Assume that 
$0<\theta<\phi<1$ and $ 0<\g_1,\g_2\leq \infty$. Set
\[
E_0=(A_0,A_1)_{\theta,\g_1}, \quad   E_1=(A_0,A_1)_{\phi,\g_2}
\]
and
\[
F_0=(B_0, B_1)_{\theta,\g_1}, \quad   F_1=(B_0,B_1)_{\phi,\g_2}.
\]
If $(B_0, F_1)$ is $K$-closed in $(A_0, E_1)$ and $(F_0, B_1)$ is $K$-closed  in $(E_0, A_1)$, then $(B_0, B_1)$ is $K$-closed in $(A_0, A_1)$.
\end{theorem}

For convenience, we will make use of  the following two definitions.

\begin{definition}
A family of quasi-Banach spaces $\{X_{p,\g}\}_{p \in (0, \infty], \g\in (0,\infty]}$ is said to 
form a \emph{real interpolation scale} if for every $0<p<q \leq \infty$,  $0<\g_1,\g_2\leq \infty$,  $0<\theta<1$, and $1/r=(1-\theta)/p +\theta/q$,
\[
X_{r,\g}= (X_{p,\gamma_1}, X_{q,\gamma_2})_{\theta, \gamma}.
\]
\end{definition}

\begin{definition}\label{K-family} 
Assume that  a family $\{X_{p,\g}\}_{p \in (0, \infty], \g\in (0,\infty]}$  forms a real  interpolation scale and  for each $0<p\leq \infty$ and  $0<\g\leq \infty$, $Y_{p,\g}$ is a subspace of  $X_{p,\g}$.  
Let $I$ be a subinterval of  $(0,\infty]$. We say that the family  $ \{Y_{p,\g}\}_{p \in (0, \infty], \g\in (0,\infty]}$  is \emph{$K$-closed  in  the family $\{X_{p,\g}\}_{p \in (0, \infty], \g\in (0,\infty]}$ on the interval $I$} if for every $p, q \in I$ and $\g_1, \g_2  \in (0,\infty]$,
the couple $(Y_{p,\g_1}, Y_{q,\g_2})$ is $K$-closed in  the couple $(X_{p,\g_1}, X_{q,\g_2})$.
\end{definition}

According to \eqref{Lp}, the family of Lorentz-spaces is an example of  a family that forms a real interpolation scale. 
 
 The following result may be viewed as an extension of Theorem~\ref{Wolff-K} to  families of quasi-Banach spaces.
\begin{proposition}\label{union}
Let  $ \{X_{p,\g}\}_{p \in (0, \infty], \g\in (0,\infty]}$ and $ \{Y_{p,\g}\}_{p \in (0, \infty], \g\in (0,\infty]}$  be  two families  such that  each forms a  real  interpolation scale.  Assume that  the family $ \{Y_{p,\g}\}_{p \in (0, \infty], \g\in (0,\infty]}$  is $K$-closed  in the family  $\{X_{p,\g}\}_{p \in (0, \infty], \g\in (0,\infty]}$ on  two different  intervals $I$ and $J$. If
$| I \cap J |>1$, then  the family $ \{Y_{p,\g}\}_{p \in (0, \infty], \g\in (0,\infty]}$  is $K$-closed  in  the family $\{X_{p,\g}\}_{p \in (0, \infty], \g\in (0,\infty]}$  on the interval $I \cup J$.
\end{proposition}

\begin{proof}
 We may assume that $I$ and $J$ are closed intervals. As $|I\cap J|>1$,   assume that $\sup I >\inf J $ and $I\cap J=[w_1, w_2]$ where $w_1=\inf J$ and $w_2=\sup  I$. Fix $p \in I \setminus J$ and  $q \in J\setminus  I$. Then $p<w_1<w_2<q$.  Let $0<\g_1,\g_2\leq \infty$.  We need to verify that $(Y_{p,\g_1}, Y_{q,\g_2})$ is $K$-closed in $(X_{p,\g_1}, X_{q,\g_2})$.

Let $1/w_j = (1-\theta_j)/p +\theta_j/q$ for $j=1,2$. One can easily see that $0<\theta_1<\theta_2<1$. We have  by assumption that
\[
X_{w_1,\g_1} = (X_{p,\g_1}, X_{q,\g_2})_{\theta_1, \g_1}\ \ \text{and}\ \  X_{w_2,\g_2}= (X_{p,\g_1}, X_{q,\g_2})_{\theta_2, \g_2}.
\]
Similarly, we also have from the assumption that 
\[
Y_{w_1,\g_1} = (Y_{p,\g_1}, Y_{q,\g_2})_{\theta_1, \g_1}\ \ \text{and}\ \  Y_{w_2,\g_2}= (Y_{p,\g_1}, Y_{q,\g_2})_{\theta_2, \g_2}.
\]

Since both $p$ and $w_2$ belong to $I$,  we have  that $(Y_{p,\g_1}, Y_{w_2,\g_2})$  is $K$-closed in 
$(X_{p,\g_1}, X_{w_2,\g_2})$. Similarly, since  both $w_1$ and  $q$ belong to  $J$, we have 
$(Y_{w_1,\g_1}, Y_{q,\g_2})$  is $K$-closed in 
$(X_{w_1,\g_1}, X_{q,\g_2})$.

We apply Theorem~\ref{Wolff-K} with $A_0= X_{p,\g_1}$,  $A_1=X_{q,\g_2}$, 
$B_0= Y_{p,\g_1}$,  $B_1=Y_{q,\g_2}$, $E_0=X_{w_1,\g_1}$, $E_1=X_{w_2,\g_2}$,
$F_0=Y_{w_1,\g_1}$, and $F_1=Y_{w_2,\g_2}$,  to conclude that 
$(Y_{p,\g_1}, Y_{q,\g_2})$ is $K$-closed in $(X_{p,\g_1}, X_{q,\g_2})$ as desired.
\end{proof}

\subsection{Noncommutative martingale Hardy spaces}\label{martingale}
By a filtration  $(\M_n)_{n \geq 1}$,   we mean an
increasing sequence of von Neumann subalgebras of ${\M}$
whose union  is w*-dense in
$\M$.  Throughout, we will work with a fixed filtration $(\M_n)_{n\geq 1}$. For  every $n\geq 1$,  we assume  further that there is a trace preserving conditional expectation $\E_n$
from ${\M}$ onto  ${\M}_n$. This is the case  if  for every $n\geq 1$, the restriction of  the trace $\T$  on $\M_n$ is semifinite. It is well-know that for $1\leq p<\infty$,  the $\E_n$'s  extend to be bounded projections from $L_p(\M,\T)$ onto $L_p(\M_n,\T|_{\M_n})$. In particular, they are well-defined  on $L_1(\M) +\M$.

\begin{definition}
A sequence $x = (x_n)_{n\geq 1}$ in $L_1(\M)+\M$ is called \emph{a
noncommutative martingale} with respect to the filtration  $({\M}_n)_{n \geq
1}$ if  for every $n \geq 1$,
\[
\E_n (x_{n+1}) = x_n.
\]
\end{definition}

Let $E$ be a symmetric quasi-Banach function space  and $x=(x_n)_{n\geq 1}$ be a martingale. If  for every $n\geq 1$, $x_n \in E(\M_n)$, then we  say that  $(x_n)_{n\geq 1}$ is an $E(\M)$-martingale.  In this case, we set
\begin{equation*}\| x \|_{E(\M)}= \sup_{n \geq 1} \|
x_n \|_{E(\M)}.
\end{equation*}
If $\| x \|_{E(\M)} < \infty$, then $x$   will be called
a bounded $E(\M)$-martingale.

For a martingale $x=(x_n)_{n\geq 1}$, we set $dx_n=x_n-x_{n-1}$  for $n\geq 1$ with the usual convention that $x_0=0$. The sequence $dx=(dx_n)_{n\geq 1}$ is called the \emph{martingale difference sequence} of $x$. A martingale $x$  is called a \emph{finite martingale} if there exists $N$ such that $dx_n=0$ for all $n\geq N$.

Let us now  review some basic  definitions related to   martingale Hardy  spaces associated to   noncommutative   symmetric spaces. 

Following \cite{PX}, we define  the  \emph{column square functions} of a given
  martingale $x = (x_k)$ by setting:
 \[
 S_{c,n} (x) = \Big ( \sum^n_{k = 1} |dx_k |^2 \Big )^{1/2}, \quad
 S_c (x) = \Big ( \sum^{\infty}_{k = 1} |dx_k |^2 \Big )^{1/2}\,.
 \]
The conditioned versions were introduced in \cite{JX}. For  a given $L_2(\M) +\M$-martingale $(x_k)_{k\geq 1}$, we set 
  \[
 s_{c,n} (x) = \Big ( \sum^n_{k = 1} \E_{k-1}|dx_k |^2 \Big )^{1/2}, \quad
 s_c (x) = \Big ( \sum^{\infty}_{k = 1} \E_{k-1}|dx_k |^2 \Big )^{1/2}
 \]
(here, we take $\E_0=\E_1$). The operator $s_c(x)$ is  called  the \emph{column conditioned square function} of $x$.
  For convenience, we will use the notation 
  \[
  \cal{S}_{c,n}(a)= \Big ( \sum^n_{k = 1} |a_k |^2 \Big )^{1/2}, \quad
 \cal{S}_c (a) = \Big ( \sum^{\infty}_{k = 1} |a_k |^2 \Big )^{1/2}
 \]
and  
 \[
  \s_{c,n} (b) = \Big ( \sum^{n}_{k = 1} \E_{k-1}|b_k |^2 \Big )^{1/2}, \quad
 \s_c (b) = \Big ( \sum^{\infty}_{k = 1} \E_{k-1}|b_k |^2 \Big )^{1/2}
 \]
 for sequences $a=(a_k)_{k\geq 1}$ in $L_1(\M)+\M$  and $b=(b_k)_{k\geq 1}$ in $L_2(\M) +\M$   that are not necessarily  martingale difference sequences.  
 It is worth pointing out that the infinite sums of positive operators stated above may not always make sense as operators  but we only consider below special cases where they do converge in the sense of  the topology in measure.
 
 We will now  describe various  noncommutative martingale Hardy spaces associated with symmetric quasi-Banach  function spaces.
 
 We consider first the class of Hardy spaces associated with square functions.
Assume  that $E$ is a symmetric  quasi-Banach function space. We denote by $\cal{F}_E$ the collection  of all finite martingales in $E(\M) \cap \M$.  
For  $x=(x_k)_{k\geq 1} \in \cal{F}_E$, we set:
\[
\big\| x  \big\|_{\mathcal{H}_E^c}= \big\| {S}_c(x) \big\|_{E(\M)} .\]
Then $(\cal{F}_E, \|\cdot\|_{\cal{H}_E^c})$ is a quasi-normed space.
 If we denote by $(e_{i,j})_{i,j \geq 1}$   the family of unit matrices in $\cal{B}(\ell_2(\mathbb{N}))$, then 
 the  correspondence  $x\mapsto \sum_{k\geq 1} dx_k \otimes e_{k,1}$ maps $\cal{F}_E$ isometrically into 
a (not necessarily closed) linear subspace of $E(\M\overline{\otimes} \cal{B}(\ell_2(\mathbb{N})))$.
We define the  column Hardy space
  $\mathcal{H}_E^c (\mathcal{M})$ to  be the completion  of $(\cal{F}_E, \|\cdot\|_{\cal{H}_E^c})$.   It then follows that 
$\H_E^c(\M)$  embeds isometrically into a closed subspace of  the quasi-Banach space $E(\M\overline{\otimes} \cal{B}(\ell_2(\mathbb{N})))$. 

We remark that   using the above definition for $L_p$ where $0<p<\infty$, we recover the definition of $\H_p^c(\M)$ as defined in \cite{PX}. However, the case $p=\infty$ is not covered by the above description since it requires separability.  We define  $\H_\infty^c(\M)$ as the collection of martingales in $\M$ for which the column square functions exists in $\M$. The  norm in $\H_\infty^c(\M)$ is defined by:
\[
\big\|x\big\|_{\H_\infty^c}= \big\|S_c(x)\big\|_\infty, \quad x \in \H_\infty^c(\M).
\] 

 In the sequel, we will also make use of the  more general space $E(\M;\ell_2^c)$ which is defined as the set of all sequences  $a=(a_k)$ in $E(\M)$ for which $\cal{S}_c(a)$ exists in $E(\M)$.  In this case, we set
\[
\big\|a \big\|_{E(\M;\ell_2^c)} = \|\cal{S}_c(a)  \|_{E(\M)}.
\]
 Under the above quasi-norm, one can easily see that  $E(\M;\ell_2^c)$ is a quasi-Banach space. 
  The closed  subspace of $E(\M;\ell_2^c)$  consisting of adapted sequences will be denoted by $E^{\rm ad}(\M;\ell_2^c)$. That is,
\[
E^{\rm ad}(\M;\ell_2^c)=\Big\{  (a_n)_{n\geq 1} \in E(\M;\ell_2^c) :  \forall n\geq 1, a_n \in E(\M_n) \Big\}.
\]
Note that for  $1<p<\infty$,  it follows from the noncommutative  Stein inequality that $L_p^{\rm ad}(\M;\ell_2^c)$ is a complemented subspace of  $L_p(\M;\ell_2^c)$.  One should not expect  such complementation if one merely assumes that $E$ is  a quasi-Banach symmetric  function space.

\bigskip

We will  now discuss conditioned versions of the  spaces defined earlier.  We should remind  the reader that the   conditioned situation is more delicate. The main technical difficulty one encounters   in defining conditioned spaces lies in  the  ability of successfully defining   conditioned square functions.  When  dealing with spaces that are not linear subset of $L_2(\M) +\M$, some type of approximations are needed.

Consider the linear space $\cal{FS}$ consisting  of all $x \in \M$ such that there exists a projection $e\in \M_1$, $\T(e)<\infty$, and $x=exe$. We should note that if $\M$ is finite,  then $\cal{FS}=\M$. 
 Let  $n\geq 1$ and $0<p  \leq  \infty$.  For  $x \in \cal{FS}$,  we set
 \[
 \big\|x\big\|_{L_p^c(\M,\E_n)} =\big\|\E_n(x^*x) \big\|_{p/2}^{1/2}.
 \] 
We should  emphasize  here that if $x=exe \in\cal{FS}$  is as described above,  then $\E_n(x^*x)=e\E_n(x^*x)e$ is a well-defined operator in $\M$ and since $\T(e)<\infty$,  it follows that $\E_n(x^*x) \in L_{p/2}(\M)$ so the quasi-norm described above is well defined. For $0<p\leq \infty$,
we define the  space $L_p^c(\M,\E_n)$ to be the completion of $\cal{FS}$ with respect 
to the above quasi-norm.

 According to \cite{Ju},  for every $0<p\leq \infty$, there exists an isometric right $\M_n$-module map  $u_{n,p}: L_p^c(\M, \E_n) \longrightarrow L_p(\M_n;\ell_2^c)$ such that
\begin{equation}\label{u}
u_{n,p}(x)^* u_{n,q}(y)=\E_n(x^*y) \otimes e_{1,1},
\end{equation}
whenever $x\in L_p^c(\M;\E_n)$, $y \in L_q^c(\M;\E_n)$, and  $1/p +1/q \leq 1$.
 
An important fact about these maps is that they are independent of $p$ as the index $p$ in the presentation of \cite{Ju} was only needed to accommodate  the non-tracial case. Below, we will simply use $u_n$ for $u_{n,p}$.

\medskip

Let  $\mathfrak{F}$  be the collection of all finite sequences $a=(a_n)_{n\geq 1}$ in $\mathcal{FS}$. For $0<p\leq \infty$, we  defined the \emph{conditioned space}
  $L_p^{\rm cond}(\M; \ell_2^c)$ to be   the completion  of  the linear space $\mathfrak{F}$  with respect to the quasi-norm:
\begin{equation}\label{conditioned-norm}
\big\| a \big\|_{L_p^{\rm cond}(\M; \ell_2^c)} = \big\| \sigma_c(a)\big\|_p
\end{equation}
A  fact from 
 \cite{Ju}  that is very crucial   in the sequel is that   $L_p^{\rm cond}(\M;\ell_2^c)$ can be isometrically embedded into an $L_p$-space associated to  a semifinite von  Neumann algebra  by means of the following map: 
\[
U : L_p^{\rm cond}(\M; \ell_2^c) \to L_p(\M \overline{\otimes} \cal{B}(\ell_2(\mathbb{N}^2)))\]
defined by setting: 
\[
U((a_n)_{n\geq 1}) = \sum_{n\geq 1} u_{n-1}(a_n) \otimes e_{n,1}, \quad (a_n)_{n\geq 1} \in \mathfrak{F}.
\]
The range  of $U$ may be viewed as a double  indexed sequences $(x_{n,k})$  such that $x_{n,k} \in L_p(\M_n)$ for all $k\geq 1$. As  an operator affiliated with $\M \overline{\otimes} \cal{B}(\ell_2(\mathbb{N}^2))$,  this may be expressed as 
$\sum_{n,k} x_{n,k} \otimes e_{k,1} \otimes e_{n,1}$. 
It is immediate   from \eqref{u}  
   that if $(a_n)_{n\geq 1} \in \mathfrak{F}$ and $(b_n)_{n\geq 1} \in \mathfrak{F}$, then
\begin{equation}\label{key-identity}
U( (a_n))^* U((b_n)) =\big(\sum_{n\geq 1} \E_{n-1}(a_n^* b_n)\big) \otimes e_{1,1} \otimes e_{1,1}.
\end{equation}
In particular, if $(a_n)_{n\geq 1}  \in \mathfrak{F}$ then  $\| (a_n) \|_{L_p^{\rm cond}(\M; \ell_2^c)}=\|U( (a_n))\|_p$ and 
hence  $U$ is indeed an isometry.  

Now, we  generalize  the notion of  conditioned spaces to the setting of symmetric spaces of measurable  operators. This is done in  two steps.
 
 $\bullet$  First, we assume  that  $E$ is a symmetric quasi-Banach function space satisfying  $L_p \cap L_\infty \subseteq E\subseteq L_p +L_\infty$ for some $0<p<\infty$ and $L_p  \cap L_\infty$ is dense in $E$.  This is the case for instance when $E$ is a separable fully symmetric quasi-Banach  function space.     For a given 
   sequence $a=(a_n)_{n\geq 1} \in \mathfrak{F}$, we set:
\[
\big\| (a_n) \big\|_{E^{\rm cond}(\M; \ell_2^c)} = \big\| \sigma_c(a)\big\|_{E(\M)}=\big\| U( (a_n))\big\|_{E(\M \overline{\otimes} \cal{B}(\ell_2(\mathbb{N}^2)))}.
\]
This is well-defined and  induces   a  quasi-norm  on the linear space $\mathfrak{F}$.  We define   the quasi-Banach space $E^{\rm cond}(\M;  \ell_2^c)$ to be the completion of  the quasi normed space $(\mathfrak{F}, \|\cdot\|_{E^{\rm cond}(\M; \ell_2^c)} )$.  The space $E^{\rm cond}(\M;\ell_2^c)$ will be called \emph{the  column conditioned space associated with $E$}.
It is clear that  
 $U$ extends to an isometry from  $E^{\rm cond}(\M;\ell_2^c)$ into $E(\M \overline{\otimes} \cal{B}(\ell_2(\mathbb{N}^2)))$ which we will still denote by $U$.

$\bullet$  Assume now that $E\subseteq L_p +L_q$  for some $0<p,q<\infty$ that is not necessarily separable. Since $L_p +L_q$ is separable, we can define $(L_p +L_q)^{\rm cond}(\M;\ell_2^c)$ using the previous case. We set
\[
E^{\rm cond}(\M; \ell_2^c)=\Big\{ x \in (L_p +L_q)^{\rm cond}(\M;\ell_2^c) : U(x) \in E(\M \overline{\otimes} \cal{B}(\ell_2(\mathbb N^2)) )\Big\}
\]
 equipped  with the  quasi-norm: 
 \[
 \big\|x \big\|_{E^{\rm cond}(\M; \ell_2^c)} =\big\| U(x) \big\|_{E(\M \overline{\otimes} \cal{B}(\ell_2(\mathbb{N}^2)))}. 
 \]
We refer to \cite{Ran-Int} for the fact that $E^{\rm cond}(\M; \ell_2^c)$ is a quasi-Banach space and $U$ is an isometry from  $E^{\rm cond}(\M; \ell_2^c)$ into $E(\M \overline{\otimes} \cal{B}(\ell_2(\mathbb N^2)) )$.
Moreover, $E^{\rm cond}(\M; \ell_2^c)$ is independent of  the specific choice of $p$ and $q$ and if $E$ is separable then 
 the above definition coincides  with the one described in the previous case.

  \bigskip
 
  We now recall the construction of column  conditioned martingale Hardy spaces.
  As in the conditioned spaces,  we describe the noncommutative  conditioned Hardy spaces in steps. Let $\mathfrak{F}(M)$ be the collection of all finite martingale  $(x_n)_{1\leq n \leq N}$   for which  $ x_N  \in \cal{FS}$.  

$\bullet$  First, assume that $E\subseteq  L_2 + L_\infty$. In this case, column conditioned square functions are well-defined for  bounded martingales in $E(\M)$. We define
$\h_E^c(\M)$ to be the collection of all  bounded martingale $x$ in $E(\M)$ for which $s_c(x) \in E(\M)$. We equip $\h_E^c(\M)$ with the norm:
\[
 \big\|x \big\|_{\h_E^c} = \big\|s_c(x)\big\|_{E(\M)}.
 \]
 One can easily verify that $(\h_E^c(\M), \|\cdot\|_{\h_E^c})$ is complete. In particular, 
 $\h_p^c(\M)$ is defined in this fashion for $2\leq p\leq \infty$.

$\bullet$  Next, we consider  quasi-Banach space  $E$  such  that  $L_p\cap L_\infty$ is dense in $E$ for some $0<p<\infty$. This is the case if $E$ is separable.
 Let $x \in \mathfrak{F}(M)$.  As noted above, $s_c(x) \in L_p(\M) \cap \M$. In particular, $s_c(x) \in E(\M)$.   We equip $\mathfrak{F}(M)$  with the quasi-norm
 \[
 \big\|x \big\|_{\h_E^c} = \big\|s_c(x)\big\|_{E(\M)}=\big\|(dx_n)\big\|_{E^{\rm cond}(\M;\ell_2^c)}.
 \]
 The column conditioned Hardy space  $\h_E^c(\M)$ is the completion of $(\mathfrak{F}(M), \|\cdot\|_{\h_E^c})$.   Clearly, the map  $x\mapsto (dx_n)$ (from $\mathfrak{F}(M)$ into 
 $\mathfrak{F}$)  extends to  be  an isometry from $\h_E^c(\M)$ into $E^{\rm cond}(\M;\ell_2^c)$ which we denote by $\mathcal{D}_c$. In particular, $\h_E^c(\M)$ is isometrically isomorphic to a  subspace of $E(\M\overline{\otimes}\cal{B}(\ell_2(\mathbb{N}^2)))$ via the isometry $U\cal{D}_c$. This  case provides  in particular the formal definition of $\h_p^c(\M)$ for $0<p<2$ or more generally  $\h_{L_p +L_q}^c(\M)$ for $0<p<q<\infty$.
 
  We should note here that if $L_2 \cap L_\infty$ is dense in $E$ and $E\subseteq L_2 +L_\infty$, then the two definitions provide the same space.
 
 $\bullet$ Assume now that $E \subset L_q + L_q$ for $0<p <q<\infty$. As in the case of conditioned spaces, we set
 \[
 \h_E^c(\M)=\Big\{ x \in \h_{L_p +L_q}^c(\M) : U\cal{D}_c(x) \in E(\M \overline{\otimes} \cal{B}(\ell_2(\mathbb N^2)) )\Big\}
 \]
equipped with the quasi-norm: 
 \[
 \big\|x \big\|_{\h_E^c} =\big\| U\cal{D}_c(x) \big\|_{E(\M \overline{\otimes} \cal{B}(\ell_2(\mathbb{N}^2)))}. 
 \]
Since the operator $U$ is independent of the index, one can easily see that the  space $\h_E^c(\M)$  is independent of  the specific choice of $p$ and $q$. Moreover, one can verify as in the case of conditioned spaces that the quasi-normed space  $(\h_E^c(\M), \|\cdot\|_{\h_E^c})$ is complete. Furthermore, if   $E$ is such that $L_2 \cap L_\infty$ is dense in $E$ then $\h_E^c(\M)$ coincides with the one defined through completion considered in the second bullet. We refer to \cite{Ran-Int, Ran-Wu-Xu} for more details.

In the sequel,  noncommutative column Hardy spaces associated with  the Lorentz space $L_{p,q}$ will be denoted by $\H_{p,q}^c(\M)$ and $\h_{p,q}^c(\M)$.

For convenience, we record below essential properties of $\H_E^c(\M)$,  $E^{\rm cond}(\M;\ell_2^c)$, and $\h_E^c(\M)$ that are relevant for our purpose. 

\begin{proposition}\label{complemented}
Assume that $E$ be a symmetric Banach function space  so that $E \in {\rm Int}(L_p, L_q)$ for $0<p<q<\infty$. Then:
\begin{enumerate}[{\rm(i)}]
\item $U: E^{\rm cond}(\M;\ell_2^c) \to E(\M \overline{\otimes} \cal{B}(\ell_2(\mathbb{N}^2)))$ is an isometric embedding.
\item $U\cal{D}_c:  \h_E^c(\M) \to E(\M \overline{\otimes} \cal{B}(\ell_2(\mathbb{N}^2)))$ is an isometric embedding.

\noindent Moreover, if $1<p<q<\infty$, then
\item $\h_E^c(\M)$ and $E^{\rm cond}(\M;\ell_2^c)$ are complemented  in $E(\M \overline{\otimes} B(\ell_2(\mathbb{N}^2)))$;
\item $\H_E^c(\M)$ is complemented in $E(\M; \ell_2^c)$.
\end{enumerate}
Similarly, $U: L_\infty^{\rm cond}(\M;\ell_2^c) \to \M \overline{\otimes} \cal{B}(\ell_2(\mathbb{N}^2))$  and $U\cal{D}_c: \h_\infty^c(\M) \to \M \overline{\otimes} \cal{B}(\ell_2(\mathbb{N}^2))$ are isometric embeddings.
\end{proposition}

\smallskip

All definitions and statements above admit corresponding row versions by passing to adjoints. For instance, the row square function of a martingale $x$ is defined as $S_r(x)=S_c(x^*)$, and the row Hardy space $\mathcal{H}_p^r (\mathcal{M})$ consists of all martingales $x$ such that $x^*\in \mathcal{H}_p^c (\mathcal{M})$.
  \smallskip
  
A third type of Hardy spaces that we will use in the sequel are the diagonal Hardy spaces. For $1\leq p\leq \infty$, we recall that the diagonal Hardy space $\h_p^d(\M)$  is the subspace of $\ell_p(L_p(\M))$ consisting of martingale difference sequences. This definition can be easily extended to the case of symmetric spaces  by setting   $\h_E^d(\M)$ as   the space of all martingales whose martingale difference sequences belong to $E(\M \overline{\otimes} \ell_\infty)$, equipped with the norm $\|x\|_{\h_E^d} := \|(dx_n)\|_{E(\M \overline{\otimes} \ell_\infty)}$.  We will denote by $\cal{D}_d$ the isometric embedding of $\h_E^d$ into $E(\M \overline{\otimes} \ell_\infty)$ given by $x \mapsto (dx_n)_{n\geq 1}$. 

We may also defined the mixed Hardy spaces by setting for $E\in {\rm Int}(L_p, L_2)$ 
with  $0<p<2$, 
\[
\H_E(\M)=\H_E^c(\M) + \H_E^r(\M) \ \  \text{and} \ \  \h_E(\M)=\h_E^c(\M) + \h_E^r(\M) + \h_E^d(\M),\]
while for $F\in {\rm Int}(L_2, L_q)$ 
with  $2<q \leq \infty$, 
\[
\H_F(\M)=\H_F^c(\M) \cap \H_F^r(\M) \ \ \text{and}\ \  \h_F(\M)=\h_F^c(\M) \cap  \h_F^r(\M) \cap \h_F^d(\M).
\]


\section{Interpolations of martingale Hardy spaces: the real method}

\subsection{Real interpolations of the couple $\bf{\big(\h_p^c , \h_\infty^c\big)}$ for ${\bf 0<p<\infty}$}

In this subsection, we prove  the primary result of the paper. A version of P. Jones's theorem for noncommutative column  conditioned martingale Hardy spaces formulated here in the form of $K$-closed couple.
\begin{theorem}\label{main-closed} 
For a given  $0<p <\infty$, the compatible couple $(\h_p^c(\M), \h_\infty^c(\M))$ is $K$-closed in
the   couple $(L_p(\mathcal{N}), \N))$ where $\N=\M \overline{\otimes} \cal{B}(\ell_2(\mathbb{N}^2))$.
\end{theorem}
 
Recall that $\h_p^c(\M)$ and $\h_\infty^c(\M)$ embed  isometrically into $L_p(\N)$ and $\N$ respectively. Thus,  we may view $\h_p^c(\M)$ (resp. $\h_\infty^c(\M)$) as a subspace of $L_p(\N)$ (resp. $\N$). Moreover, all spaces embed continuously into the topological vector space $\widetilde{\N}$, the space of all $\widetilde{\T}$-measurable operators affiliated with the von Neumann algebra $\N$ (here $\widetilde{\T}$ denotes the natural trace on $\N$). In particular, $(\h_p^c(\M), \h_\infty^c(\M))$ is a compatible couple which is a subcouple of $(L_p(\N), \N)$.

\medskip

The main step in the proof of the theorem is  for  the couple $(\h_2^c(\M),  \h_\infty^c(\M))$.  The following  constitutes the  decisive ingredient in  our  argument.

\begin{proposition}\label{main}
There exists a constant $C>0$ so that for every $x \in \h_2^c(\M) +\h_\infty^c(\M)$ and $t>0$, the following inequality holds:
\[
K\big(x, t; \h_2^c(\M), \h_\infty^c(\M) \big)\leq C \Big(\int_0^{t^2} \big(\mu_u(s_c(x))\big)^2 \ du\Big)^{1/2}.
\]
\end{proposition}

\begin{proof}
Let $x \in \h_2^c(\M) + \h_\infty^c(\M)$ and $t>0$.   Fix $\epsilon>0$ and set 
\[
\lambda= \frac{2+\epsilon}{t} \Big(\int_0^{t^2} (\mu_{u}(s_c(x)))^2 \ du\Big)^{1/2}.
\]
The   two step construction used below is modeled after a similar decomposition considered in  \cite{Chen-Ran-Xu}  which was inspired by an idea  initially used  in \cite{PR}.  It is primarily based on the use of the so-called Cuculescu's  projections (\cite{Cuc}). Submajorizations also play prominent role in the argument.

$\bullet$ Step~1. (Initial construction)

We apply the construction of Cuculescu's projections to the submartingale $(s_{c, k}^2 (x))_{k \ge 1}$  and the  parameter $\lambda^2$. That is,  we start with  $q_0 ={\bf 1}$ and for $k\geq 1$, we set
\[
q_k := q_{k-1} \ch_{[0,\lambda^2]} \big(
q_{k-1}  s_{c,k}^2(x) q_{k-1} \big)=\ch_{[0,\lambda^2]}\big(
q_{k-1} s_{c,k}^2(x)  q_{k-1} \big)q_{k-1}.
\]
Then   $(q_k)_{k\ge 1}$ is a decreasing sequence  of projections  in $\M$ satisfying the following properties:
\begin{enumerate}
\item $q_k \in \M_{k-1}$ for every $k\geq 1$;
\item  $q_k$ commutes with $q_{k-1} s_{c, k}^2 (x) q_{k-1}$ for all $k\geq 1$;
\item $q_k s_{c, k}^2 (x) q_k \le \lambda^2 q_k$ for all $k\geq 1$;
\item  if we set $q= \bigwedge_{k \ge 1} q_k$,  then 
$
\lambda^2({\bf 1}- q)  \leq \sum_{k\geq 1} (q_{k-1}-q_{k}) s_{c,k}^2(x) (q_{k-1}-q_k)$.
 \end{enumerate}
These facts on Cuculescu's projections are now standard. Verifications can be found in \cite{Chen-Ran-Xu, PR, Ran15}.
 
 We consider the following  two martingale  difference sequences: for $k\geq 1$,
\begin{equation}\label{a}
 d\alpha_k = d x_k q_k\  \ \text{and}\   \ d\beta_k = d x_k q_{k-1}.
\end{equation}
We denote  the corresponding   martingales  by $\a = (\a_k)_{k \ge 1}$  and $\b=(\b_k)_{k\geq 1}$  respectively. We only use $\a$ in the construction  below but $\b$  will be needed in order to deduce some properties of $\a$. We record the next lemma for further use.
\begin{lemma}\label{lemma1}
 The martingales $\a$ and $\b$ satisfy the following properties:
\begin{enumerate}[$\rm (i)$]
\item $\| \E_{k-1} ( |d \a_k|^2 )\|_{\infty} \le \lambda^2$ for all $k \ge 1$;
\item   $s_c^2(\b)  \prec\prec  s_c^2(x)$;
\item $s_c^2(\a)  \prec\prec 4 s_c^2(x)$.
\end{enumerate}
\end{lemma}
\begin{proof}
The first item   can be deduced as follows: from \eqref{a}, we have for $k\geq 1$, 
\begin{equation*}\begin{split}
\E_{k-1} ( |d \a_k|^2 ) &= \E_{k-1} (q_k | d x_k |^2 q_k )\\
&=q_k \E_{k-1} ( | d x_k |^2  ) q_k\\
&=q_k[ s_{c,k}^2(x) -s_{c,k-1}^2(x)]q_k\\
&\leq q_k s^2_{c, k} (x) q_k
\leq \lambda^2 q_k
\end{split}
\end{equation*}
where in the second identity, we use the fact that $(q_k)_{k\geq 1}$ is a predictable sequence. The estimate clearly implies  item~(i).

For the second item, we have from the definition  of $\b$ that  for $m\geq 1$  (with $s_{c,0}(x)=0$),
\begin{align*}
s_{c,m}^2(\b) &= \sum_{k=1}^m q_{k-1} \E_{k-1}(|dx_k|^2) q_{k-1}\\
&=\sum_{k=1}^m q_{k-1} [s_{c,k}^2(x)- s_{c,k-1}^2(x)]q_{k-1}\\
&=\sum_{k=1}^m q_{k-1}s_{c,k}^2(x)q_{k-1} - \sum_{k=1}^m q_{k-1}s_{c,k-1}^2(x)q_{k-1}.
\end{align*}
Performing some indexing shift, we obtain that
\[
s_{c,m}^2(\b)= q_{m-1} s_{c,m}^2(x) q_{m-1} +\sum_{k=1}^{m-1} \big(q_{k-1} s_{c,k}^2(x)q_{k-1}- q_{k} s_{c,k}^2(x)q_{k}\big).
\]
  From the fact  that $q_k$ commutes with $q_{k-1} s_{c, k}^2 (x) q_{k-1}$,  we deduce that
\begin{align*}
s_{c,m}^2(\b) &=q_{m-1} s_{c,m}^2(x)  q_{m-1}+\sum_{k=1}^{m-1} (q_{k-1}-q_k)s_{c,k}^2(x) (q_{k-1}-q_k)\\
&\leq q_{m-1} s_{c}^2(x) q_{m-1}+\sum_{k=1}^{m-1} (q_{k-1}-q_k)s_c^2(x) (q_{k-1}-q_k).
\end{align*}
Note that the finite family of projections $\{q_{k-1}-q_k: 1\leq k\leq m-1\} \cup \{q_{m-1}\}$  is mutually disjoint. We may  deduce  from \eqref{sub-diagonal}  that  for every $m\geq 1$, $s_{c,m}^2(\b) \prec\prec s_c^2(x)$. We arrive at the desired conclusion by noticing that for every  $w>0$, the monotone convergence theorem gives:
\[
\int_0^w \mu_u(s_c^2(\b))\ du =\lim_{m\to \infty} \int_0^w \mu_u(s_{c,m}^2(\b))\ du \leq 
\int_0^w \mu_u(s_c^2(x))\ du. 
\]
That is, $s_{c}^2(\b) \prec\prec s_c^2(x)$.

For the last item,  we note first that a straightforward computation  gives:
\begin{align*}
s_c^2(\b-\a) &=\sum_{k\geq 1} (q_{k-1}-q_k) [s_{c,k}^2(x)-s_{c,k-1}^2(x)] (q_{k-1}-q_k) \\ 
&\leq  \sum_{k\geq 1} (q_{k-1}-q_k)s_c^2(x) (q_{k-1}-q_k)
\end{align*}
 and therefore we have from \eqref{sub-diagonal} that  $s_c^2(\b-\a) \prec\prec s_c^2(x)$.  Using the elementary inequality $|a +b|^2 \leq 2|a|^2 + 2|b|^2$ for operators $a$ and $b$, we can then  conclude  from \eqref{sub-sum} and  the second item that
 \[
 s_c^2(\a) \leq 2s_c^2(\b-\a) +2s_c^2(\b) \prec\prec  4 s_c^2(x). 
 \]
 The lemma is verified.
\end{proof}

$\bullet$  Step~2. The decomposition that gives  the desired estimate on the $K$-functional.

First, we  apply the construction of  Cuculescu's projections to  the submartingale $(s_{c, k}^2 (\a))_{k \ge 1}$ and the parameter $\lambda^2$ where $\a$ is the martingale from \eqref{a}.  
That is,  setting   $\pi_0 ={\bf 1}$ and for $k\geq 1$, we define:
\[
\pi_k := \pi_{k-1} \ch_{[0,\lambda^2]} \big(
\pi_{k-1}  s_{c,k}^2(\a) \pi_{k-1} \big)=\ch_{[0,\lambda^2]}\big(
\pi_{k-1} s_{c,k}^2(\a)  \pi_{k-1} \big)\pi_{k-1}.
\]
Then   $(\pi_k)_{k\ge 1}$ is a decreasing sequence  of projections  in $\M$. As before, it satisfies the following properties:
\begin{enumerate}
\item $\pi_k \in \M_{k-1}$ for every $k\geq 1$;
\item  $\pi_k$ commutes with $\pi_{k-1} s_{c, k}^2 (\a) \pi_{k-1}$ for all $k\geq 1$;
\item $\pi_k s_{c, k}^2 (\a) \pi_k \le \lambda^2 \pi_k$ for all $k\geq 1$;
\item  if we set $\pi= \bigwedge_{k \ge 1} \pi_k$,  then 
$
\lambda^2({\bf 1}- \pi)  \leq \sum_{k\geq 1} (\pi_{k-1}-\pi_{k}) s_{c,k}^2(\a) (\pi_{k-1}-\pi_k)$.
 \end{enumerate}

 Next, we define two martingales $y$ and  $z$ by setting:
\begin{equation}\label{z}
z= \sum_{k \geq 1} d \a_k \pi_{k - 1}=\sum_{k\geq 1}dx_kq_k\pi_{k-1} \ \ \text{and} \ \ 
y=\sum_{k\geq 1}dx_k({\bf 1}-q_k\pi_{k-1}).
\end{equation}
Clearly, we have the decomposition:
\[
x= y +z.
\]
We will show that this decomposition provides the desired estimate on the $K$-functional.
We  consider first the martingale $z$. We claim that   $z \in \h_\infty^c(\M)$ with 
 \begin{equation}\label{norm-z}
 \big\|z\big\|_{\h_\infty^c} \leq \sqrt{2}  \lambda.
\end{equation}
To verify \eqref{norm-z}. Fix $m\geq 1$.  From the definition of $z$, we have:
\begin{align*}
s_{c,m}^2 (z)& = \sum_{k=1}^m \E_{k-1} [\pi_{k-1} |d \a_k|^2 \pi_{k-1}]\\
& = \sum_{k= 1}^m\big ( \pi_{k-1} s^2_{c,k} (\a) \pi_{k-1} - \pi_{k-1} s^2_{c,k-1} (\a) \pi_{k-1} \big )\\
& =  \sum_{k=1}^m   \pi_{k-1} s^2_{c,k} (\a) \pi_{k-1} - \sum_{k=1}^{m-1}\pi_k s^2_{c,k} (\a) \pi_k \\
&=\pi_{m-1}s_{c,m}^2(\a) \pi_{m-1} + \sum_{k=1}^{m-1} ( \pi_{k-1} - \pi_k ) s^2_{c,k} (\a) ( \pi_{k-1} - \pi_k \big )\,,
\end{align*}
where the last equality follows from the commutativity between $\pi_k$ and $\pi_{k-1} s^2_{c,k} (\a) \pi_{k-1}$.  Recall from Lemma~\ref{lemma1}~(i) that $\| \E_{k-1} (|d \a_k|^2) \|_{\infty} \le \lambda^2$.  Using this fact,  we have
\begin{align*} 
\pi_{k-1} s_{c,k}^2(\a)\pi_{k-1} &= \pi_{k-1}[s_{c,k-1}^2(\a) +\E_{k-1} (|d \a_k|^2)]\pi_{k-1}\\
&\leq 2\lambda^2 \pi_{k-1}.
\end{align*}
Applying this inequality  with the previous estimate, we obtain that for every $m\geq 1$, 
\[
s_{c,m}^2 (z) 
\leq 2\lambda^2 \pi_{m-1} + 2\lambda^2 \sum_{k=1}^{m-1} ( \pi_{k-1} - \pi_k)
\leq 2\lambda^2{\bf 1}.
\]
Since this  holds for arbitrary $m\geq 1$, we have $s_{c}^2 (z)\leq  2\lambda^2{\bf 1}$
which   shows  that $\big\|z\big\|_{\h_\infty^c} \leq  \sqrt{2} \lambda$ and thus proving the claim.

 \medskip

We now deal with the martingale $y$.  We will estimate the norm of  $y$ in $\h_2^c(\M)$.  The following  lemma is  the most crucial  part of   the argument.
\begin{lemma}\label{trace}
The projections $q$ and $\pi$ satisfy the following property:
\[
\max\big\{ \T({\bf 1}-q), \T({\bf 1}-\pi) \big\} \leq t^2.
\]
\end{lemma}
\begin{proof}
We will  only verify that $\T({\bf 1}-\pi)\leq t^2$. The argument  for $\T({\bf 1}-q)$ is identical so we will leave the details to the reader.

Fix $w>t^2$. We claim  that  $\mu_w({\bf 1}-\pi)=0$. Assume the opposite, i.e, $\mu_w({\bf 1}-\pi)=1$.
We start with  the fact that
\begin{equation}\label{pi}
\lambda^2({\bf 1}- \pi) \leq \sum_{k\geq 1} (\pi_{k-1}-\pi_{k}) s_{c}^2(\a) (\pi_{k-1}-\pi_k).
\end{equation}
Taking generalized singular values and integrals,  inequality \eqref{pi} gives
\[
\lambda^2 \int_0^w \mu_u({\bf 1}-\pi) \ du \leq \int_0^w \mu_u\big(\sum_{k\geq 1} (\pi_{k-1}-\pi_{k}) s_{c}^2(\a) (\pi_{k-1}-\pi_k)\big) \ du.
\]
By submajorization and the fact that $\mu({\bf 1}-\pi)$ is a characteristic function and therefore is identically equal to $1$ on the interval $[0,w]$ by assumption, we have, 
\[
\lambda^2 w  \leq \int_0^w \mu_u(s_c^2(\a))\ du \leq 4 \int_0^w \mu_u(s_c^2(x))\ du
\]
where the second inequality comes from Lemma~\ref{lemma1}(iii). Using the specific value of $\lambda$, this  leads  to 
\begin{align*}
(2+\epsilon)^2w \int_0^{t^2} \mu_u(s_c^2(x))\ du  &\leq  4t^2 \int_0^{w} \mu_u(s_c^2(x))\ du \\
  &\leq  4t^2 \int_0^{t^2} \mu_u(s_c^2(x))\ du + 4t^2 \int_{t^2}^w \mu_u(s_c^2(x))\ du\\
&\le4t^2 \int_0^{t^2} \mu_u(s_c^2(x))\ du  + 4t^2 (w-t^2) \mu_{t^2}(s_c^2(x))\\
&\le(2+\epsilon)^2 t^2 \int_0^{t^2} \mu_u(s_c^2(x))\ du  + 4t^2 (w-t^2) \mu_{t^2}(s_c^2(x)).
\end{align*}
After rearrangement  and division  by $w-t^2$, we arrive at
\[
(2+\epsilon)^2 \int_0^{t^2} \mu_u(s_c^2(x))\ du \leq 4t^2 \mu_{t^2}(s_c^2(x)).
\]
Since $\mu(s_c^2(x))$ is decreasing, the left hand side of the preceding inequality  is larger than $(2+\epsilon)^2 t^2\mu_{t^2}(s_c^2(x))$ which  is a contradiction. Thus, we must have 
$\mu_w({\bf 1}-\pi)=0$. This shows that $\T({\bf 1}-\pi)\leq t^2$.
\end{proof}
Now we can estimate the norm of $y$ as follows:
first, we   split $y$ into two parts:
\begin{align*}  
y &= \sum_{k\geq 1} dx_k({\bf 1}-q_k\pi_{k-1}) \\
&=\sum_{k\geq 1} dx_k({\bf 1}- q_k) + \sum_{k\geq 1} d\a_k({\bf 1}-\pi_{k-1})\\
&=(x-\a) + (\a -z).
\end{align*}
  Recall  that  the martingale $z$  is such that $dz_k = d\a_k\pi_{k-1}$ for $k\geq 1$. We will use below that  $s_c^2(z) \prec\prec s_c^2(\a)$. The verification of this fact is identical to the martingale $\b$  in Lemma~\ref{lemma1}(ii) so we omit the details.

We now proceed with  the estimation of the norm:
\begin{align*}
\big\|y \big\|_{\h_2^c} &\leq \big\|x-\a\big\|_{\h_2^c} + \big\|\a-z\big\|_{\h_2^c}\\
&=\big\| s_c^2(x-\a) \big\|_1^{1/2} + \big\| s_c^2(\a-z)\big\|_1^{1/2}\\
&=\Big( \int_0^\infty \mu_u(s_c^2(x-\a))\ du \Big)^{1/2} +\Big( \int_0^\infty \mu_u(s_c^2(\a-z))\ du \Big)^{1/2}.
\end{align*}
The important  fact here is that  the operator $s_c^2(x-\a)$ (resp. $s_c^2(\a-z)$) is supported by 
the projection $({\bf 1}- q)$ (resp. $({\bf 1} -\pi)$). By properties of generalized singular values and Lemma~\ref{trace},  we have $\mu_u(s_c^2(x-\a))=\mu_u(s_c^2(\a-z))=0$ for  every $u>t^2$. Therefore,  the last estimate can be refined as:
\[
\big\|y \big\|_{\h_2^c} \leq \Big( \int_0^{t^2} \mu_u(s_c^2(x-\a))\ du \Big)^{1/2} +\Big( \int_0^{t^2} \mu_u(s_c^2(\a-z))\ du \Big)^{1/2}.
\]
We estimate the two integrals on the right hand side separately. Using the elementary inequality $s_c^2(x-\a) \leq 2 s_c^2(x) +2 s_c^2(\a)$ and   the submajorization stated in Lemma~\ref{lemma1}(iii), we have:
\begin{align*}
\int_0^{t^2} \mu_u(s_c^2(x-\a))\ du &\leq 2\int_0^{t^2} \mu_u(s_c^2(x))\ du  +2\int_0^{t^2} \mu_u(s_c^2(\a))\ du \\
&\leq 10 \int_0^{t^2} \mu_u(s_c^2(x))\ du.
\end{align*}
Similarly, the second integral can be estimated as follows:
\begin{align*}
\int_0^{t^2} \mu_u(s_c^2(\a-z))\ du &\leq 2\int_0^{t^2} \mu_u(s_c^2(\a))\ du  +2\int_0^{t^2} \mu_u(s_c^2(z))\ du \\
&\leq 4 \int_0^{t^2} \mu_u(s_c^2(\a))\ du\\
&\leq 16  \int_0^{t^2} \mu_u(s_c^2(x))\ du,
\end{align*}
where in the second inequality  we have used the observation stated earlier that $s_c^2(z) \prec\prec s_c^2(\a)$.
Combining the   estimates on the two  integrals above, we get 
\begin{equation}\label{norm-y}
\big\|y\big\|_{\h_2^c}\leq  (\sqrt{10} +4)  \Big( \int_0^{t^2} \mu_u(s_c^2(x))\ du \Big)^{1/2}.
\end{equation}

We can now estimate the $K$-functional using the decomposition $x=y+z$. Indeed,  by 
combining \eqref{norm-z} and \eqref{norm-y}, we have
\begin{align*}
K\big(x, t; \h_2^c(\M), \h_\infty^c(\M)\big) &\leq  \big\|y \big\|_{\h_2^c} + t \big\|z\big\|_{\h_\infty^c} \\
&\leq   (\sqrt{10} + 4 + 2\sqrt{2}  +\sqrt{2}\epsilon)  \Big( \int_0^{t^2} \mu_u(s_c^2(x))\ du \Big)^{1/2}.
\end{align*}
Since $\epsilon$ is arbitrary, we may conclude that
\[
K\big(x, t; \h_2^c(\M), \h_\infty^c(\M)\big) \leq (\sqrt{10} +4 +2\sqrt{2})  \Big( \int_0^{t^2} \mu_u(s_c^2(x))\ du \Big)^{1/2}. 
\]
The proof  of Proposition~\ref{main} is complete.
\end{proof}

We make the following remarks about the proof of Proposition~\ref{main}. First, we would like to point out that in  order to obtain that $z\in \h_\infty^c(\M)$, it is important that the sequences $\{s_{c,k}^2(x)\}_{k\geq 1}$ and 
$\{s_{c,k}^2(\a)\}_{k\geq 1}$ are used in constructing the Cuculescu projections.
Using any exponent strictly less than $2$ is not sufficient for this  goal. Second, 
since submajorization is one of the main tool we use, it is  also essential that the operator $s_c^2(x) \in L_1(\M) +\M$. 
Therefore, our argument cannot be carried out for the couple $(\h_p^c(\M),\h_\infty^c(\M))$ when $0<p<2$. This is in strong contrast with  the stopping time and 
atomic decompositions   approach used  for classical martingales as demonstrated in \cite[Theorem~5.9]{Weisz} where actually the case $0<p\leq 1$ was given.

\medskip

We now proceed toward the proof of Theorem~\ref{main-closed}. We  recall  that for $f \in L_2 +L_\infty$ and $t>0$, we have the equivalence:
\[
K(f,t; L_2, L_\infty) \approx\Big( \int_0^{t^2} (\mu_u(f))^2 \ du \Big)^{1/2}.
\]
We refer to  \cite{Holm}  for this fact. As  is well-known,  the above equivalence   extends to the corresponding  couple of noncommutative  spaces (see \cite{PX3}).   We note from   \eqref{key-identity}
  that for $x \in \h_2^c(\M) +\h_\infty^c(\M)$,  we have $|U\cal{D}_c(x)|= s_c(x) \otimes
 e_{1,1} \otimes e_{1,1}$. Thus,  Proposition~\ref{main} can be restated in the following form: there exists a constant $C>0$ so that for every $x\in \h_2^c(\M) +\h_\infty^c(\M)$ and $t>0$,
\[
K\big(x, t; \h_2^c(\M), \h_\infty^c(\M) \big)\leq C 
K\big(U\cal{D}_c(x), t; L_2(\N), \N \big). 
\]
Consequently, Proposition~\ref{main} is equivalent to the following intermediate statement:
\begin{remark}\label{K2}
 The couple $(\h_2^c(\M), \h_\infty^c(\M))$ is $K$-closed in
the couple $(L_2(\mathcal{N}), \N)$. 
\end{remark}

As an immediate application of the preceding  remark and \eqref{Lp}, we have the following interpolation  result:
\begin{corollary}\label{h2}
Let $2\leq q<\infty$ and $0<\theta<1$. For $1/r=(1-\theta)/q$ and $0< \lambda, \g\leq \infty$,
\[
\big( \h_{q,\lambda}^c(\M), \h_{\infty}^c(\M)\big)_{\theta, \g} = \h_{r,\g}^c(\M)
\]
with equivalent norms.
\end{corollary}

On the other hand,  the corresponding result  for  finite indices is already known for the full interval $(0,\infty)$.

\begin{proposition}[\cite{Ran-Int}]\label{finite}
Let  $0<p,q <\infty$ and $0<\theta <1$. For $1/r=(1-\theta)/p + \theta/q$ and $0<\g_1,\g_2,\g \leq \infty$,
\[
\big( \h_{p,\g_1}^c(\M), \h_{q,\g_2}^c(\M)\big)_{\theta, \g} = \h_{r,\g}^c(\M)
\]
with equivalent  quasi-norms.
\end{proposition}

Combining Corollary~\ref{h2}  and Proposition~\ref{finite}, we may deduce from \cite[Lemma~3.4]{Ran-Int} that   the family $\{\h_{p,\g}^c(\M)\}_{p\in (0,\infty], \g\in (0,\infty]}$ forms a real interpolation scale on $\mathbb{R}_+\cup \{\infty\}$.  The next  statement  should be compared with \cite[Theorem~3.5]{Ran-Int} where $\bmo^c(\M)$ was used as one of the endpoints.  Although  it is implied by Theorem~\ref{main-closed}, it is needed for the proof. We  explicitly state it here  for convenience.

\begin{proposition}\label{main-interpolation} If $0<\theta <1$ and  $0<p<\infty$,    then for $1/r=(1-\theta)/p$ and $0<\lambda, \g\leq \infty$,
\[
\big( \h_{p,\lambda}^c(\M), \h_\infty^c(\M)\big)_{\theta, \g} = \h_{r,\g}^c(\M)
\]
with equivalent  quasi-norms.
\end{proposition}
Let us now return to more results on $K$-closedness that we will need for the proof of Theorem~\ref{main-closed}.
We recall  that  under  a  more restrictive  conditions on the indices,  a  $K$-closedness result  was already proved in  \cite[Corollary~3.8, Remark~3.9]{Ran-Int}. It can be stated  as follows:
\begin{lemma}[\cite{Ran-Int}]\label{K1}
Let $\nu$ be a integer with $\nu\geq 2$.  Assume that $2/(\nu +1) <p\leq 2/\nu$ and 
$p<q<2/(\nu-1)$.  Then, the couple   $(\h_p^c(\M), \h_q^c(\M))$ is $K$-closed in  the couple $(L_p(\N),L_q(\N))$.
\end{lemma}

The last piece we need  can be easily deduced   from  the complementation property recorded  in  Proposition~\ref{complemented}:

\begin{lemma}\label{K3}
If $1<p<q<\infty$, then  the couple   $(\h_p^c(\M), \h_q^c(\M))$ is $K$-closed in  the couple $(L_p(\N),L_q(\N))$.
\end{lemma}

\smallskip

\begin{proof}[End of the proof of Theorem~\ref{main-closed}]
The argument consists of mixing the  three different  intervals stated in  Remark~\ref{K2}, Lemma~\ref{K1}, and Lemma~\ref{K3}. 

For $0<p \leq \infty$ and $0<\g \leq \infty$,  let 
\[X_{p,\g} := L_{p,\g}(\N) \ \ \text{and} \ \  
 Y_{p,\g} := \h_{p,\g}^c(\M).\]
Here $L_{\infty,\g}(\N)$ (resp. $\h_{\infty,\g}^c(\M)$) is simply  $\N$ (resp. $\h_\infty^c(\M)$).

  Then, we may view $Y_{p,\g}$ as  a subspace of $X_{p,\g}$ by the isometric embedding  detailed in Subsection~\ref{martingale}.
 We have from \eqref{Lp}  that the family $\{X_{p,\g}\}_{p, \g} $ forms a real interpolation scale. Similarly, we  also have from Proposition~\ref{main-interpolation}  that the family $\{Y_{p,\g}\}_{p, \g} $ forms a real interpolation scale. Thus, the assumptions of Proposition~\ref{union} are satisfied by the two families.
 
 Consider the sequence of intervals $(I_\nu)_{\nu\geq 0}$ with 
 $I_0=(2, \infty]$,
 $I_1=(1,\infty)$,  and 
for $\nu\geq 2$, 
\[
I_\nu= (\frac{2}{\nu+1}, \frac{2}{\nu-1}).
\]

By Remark~\ref{K2} and Holmstedt formulae (see \cite[Theorem~2.1, Remark~2.1]{Holm}), the family $\{ Y_{p,\g}\}_{p,\g}$ is $K$-closed  in the family $\{ X_{p,\g}\}_{p,\g}$ on the interval $I_0$. Also, Lemma~\ref{K3} gives that the family $\{ Y_{p,\g}\}_{p,\g }$  in $K$-closed in the family $\{ X_{p,\g}\}_{p,\g}$ on the interval $I_1$.

Next,  for a given $ \nu\geq 2$, it follows from Lemma~\ref{K1} and \cite[Theorem~3.1]{Holm} that the family $\{ Y_{p,\g}\}_{p,\g}$ is $K$-closed  in the family $\{ X_{p,\g}\}_{p,\g}$ on the interval $I_\nu$. 

 We note that $I_0 \cap I_1= (2, \infty)$,  
 $I_1 \cap I_2= (1,2)$,  and for $\nu \geq 2$, $I_\nu \cap I_{\nu+1}= (2/(\nu+1), 2/\nu]$. In particular, for $\upsilon\geq 0$, $ | I_\upsilon \cap I_{\upsilon +1}|>1$.

 By  applying Proposition~\ref{union} inductively,  we deduce that the family $\{ Y_{p,\g}\}_{p,\g}$  is $K$-closed in the family $\{ X_{p,\g}\}_{p,\g}$  on the interval $\bigcup_{\upsilon=0}^\infty I_\upsilon = (0,\infty]$ which is the desired conclusion.
\end{proof}
 
\begin{remark}
The proof  of Theorem~\ref{main-closed}  actually provides the more general statement that for $0<p <q \leq \infty $ and $0<\g,\lambda \leq \infty$,  the couple   $(\h_{p,\g}^c(\M) , \h_{q,\lambda}^c(\M) )$  is $K$-closed in  the couple $(L_{p,\g}(\N) , L_{q,\lambda}(\N) )$ but this can already    be deduced from the statement of  Theorem~\ref{main-closed} by applying   \cite[Theorem~2.1]{Holm}.
\end{remark}

\begin{remark}\label{h-closed-bis} 
Assume that $0<p<q\leq \infty$. Theorem~\ref{main-closed} can be reformulated as follows:
 If  $y$ is a finite martingale  in $\mathfrak{F}(M)$ then for  every $t>0$,
\[
K\big(y,t; \h_p^c(\M), \h_q^c(\M)\big) \approx_{p,q} K\big(s_c(y), t; L_p(\M), L_q(\M)\big).
\]
The restriction to martingales from $\mathfrak{F}(M)$ is  only needed to insure that  conditioned square functions are well-defined operators but this is equivalent to Theorem~\ref{main-closed} by density.
\end{remark}

At the time of this writing, it is still open if the  interpolation result  stated in Proposition~\ref{main-interpolation}  for the couple $(\h_p^c (\M), \h_\infty^c(\M))$, $0<p<\infty$,  remains valid if   the complex interpolation method is used. See also \cite[Problem~5]{Bekjan-Chen-Perrin-Y}.

\subsection{Generalization to  the  couple $\bf{\big(\h_E^c, \h_\infty^c\big)}$}
In this subsection, we will extend Theorem~\ref{main-closed} to more general couples.   More precisely,  noncommutative column  conditioned Hardy spaces associated with general function spaces are used. It reads as follows:

\begin{theorem}\label{K-general}
Let $0<p <q<\infty$ and $E \in {\rm Int}(L_p,L_q)$.
Assume that  $\cal{F}$ is   a quasi-Banach function space  with monotone quasi-norm.   If   $E=(L_p, L_\infty)_{\cal{F} ; K}$,  then the   couple $(\h_E^c(\M), \h_\infty^c(\M))$ is $K$-closed in the  couple $(E(\N), \N)$.
\end{theorem}
Before we proceed,  we remark that the assumption $E \in {\rm Int}(L_p,L_q)$ where $0<p<q<\infty$  is needed  so that the Hardy space  $\h_E^c(\M)$  can be defined as  described in the preliminary section. The assumption implies a fortiori that $E\in {\rm Int}(L_p,L_\infty)$ and therefore the existence of the quasi-Banach function space $\cal{F}$ is given by Proposition~\ref{K-method}.

\medskip

The proof is based on  the following  more general form  of Holmstedt formula for $K$-functionals. It  will allow us to consider  a wider  class of  function spaces beyond  classical  Lorentz spaces.
\begin{proposition}[\cite{Ahmed-Kara-Reza}]\label{gen-Holm}
Let $(A_0,A_1)$ be a compatible couple of quasi-Banach spaces  and $\cal{F}$  be a quasi-Banach function space with monotone quasi-norm. If $X=  (A_0, A_1)_{\cal{F}; K}$  and  $a \in X +A_1$,  then for every  $t>0$,
\[
K(a, \rho(t) ; X, A_1)  \approx I(t,a) +\frac{\rho(t)}{t} K(a, t; A_0, A_1),
\]
where $I(t,a)=\|\ch_{(0,t)}(\cdot) K(a, \cdot \,; A_0, A_1)\|_{\cal F}$ and 
$\rho(t)\approx t\|\ch_{(t,\infty)}(\cdot)\|_{\cal{F}} + \| u\mapsto u\chi_{(0,t)}(u)\|_{\cal{F}}$.
\end{proposition}
We remark that some version of Proposition~\ref{gen-Holm} already appeared  in \cite{Mast} for  Banach spaces. We also note that Proposition~\ref{gen-Holm}   provides an alternative proof of the following fact: if   $E=(L_p, L_\infty)_{\cal{F} ; K}$ and $x \in E(\N)+\N$  then  for every $t>0$,
\begin{equation}\label{comp-K}
K\big(x,t; E(\N), \N\big) \approx_E K\big(\mu(x), t; E, L_\infty\big).
\end{equation}
This equivalence is well-known for couples  of   noncommutative $L_p$-spaces (see \cite[Corollary~2.3]{PX3} and the remark immediately after). For the case where $E$ is  a Banach space, it is a consequence of  a result on partial  retract (see \cite[Corollary~2.2]{PX3}).


\begin{proof}[Proof of Theorem~\ref{K-general}]
Let $x \in \h_E^c(\M) +\h_\infty^c(\M)$ and $t>0$. 
Set  
\[
I_1(t,x)=\big\| u\mapsto \ch_{(0,t)}(u) K(x, u; \h_p^c(\M), \h_\infty^c(\M)) \big\|_{\cal F}\]
  and 
  \[ I_2(t, U\cal{D}_c(x))=\big\| u \mapsto \ch_{(0,t)}(u) K( U\cal{D}_c(x), u; L_p(\N), \N) \big\|_{\cal F}. 
  \]
  Since $\cal{F}$ has monotone quasi-norm, it follows  from Theorem~\ref{main-closed} that:  
\begin{equation}\label{I-equiv}
I_1(t, x) \approx I_2(t, U\cal{D}_c(x)).
\end{equation}
Using \eqref{I-equiv} with  Theorem~\ref{main-closed} and Proposition~\ref{gen-Holm}, 
we may deduce that 
\begin{align*}
K\big(x,\rho(t), \h_E^c(\M), \h_\infty^c(\M)\big) &\approx  I_1(t,x) +\frac{\rho(t)}{t} K\big(x, t; \h_p^c(\M), \h_\infty^c(\M)\big) \\
&\approx I_2(t, U\cal{D}_c(x)) +\frac{\rho(t)}{t} K\big( U\cal{D}_c(x), t; L_p(\N), \N\big) \\
&\approx  K\big( U\cal{D}_c(x),\rho(t), E(\N), \N\big).
\end{align*}
To conclude the proof, we will verify that the range of  the function $\rho(\cdot)$ is $[0,\infty)$. Indeed, one can easily see that 
 $\rho(\cdot)$ is  continuous and  $\rho(0)=0$. Moreover, for $t>0$, we have from Holmstedt's formula that
 $K(\ch_{(0,t]}, u; L_p, L_\infty) \approx t\ch_{(t,\infty)}(u) +u \ch_{(0,t]}(u)$. It follows from the representation of $E$ and quasi-triangle inequality on $\cal{F}$  that
 $\|\ch_{(0,t]} \|_E \lesssim \rho(t)$. On the other hand, by assumption, $E$ is $r$-concave for some $r<\infty$. This implies that  for every $n\geq 2$,
 \[
 \big(\sum_{k=1}^n \|\ch_{(k-1,k]}\|_E^r \big)^{1/r} \lesssim \| \ch_{(0, n]} \|_E.
 \]
Since $E$ is symmetric, the  left hand side is equal to  $n^{1/r}\| \ch_{(0,1]}\|_E$. We  deduce that for every $t>1$,
\[
(t-1)^{1/r}\|\ch_{0,1]}\|_E \lesssim \rho(t).
\] 
 This clearly implies that 
  $\lim_{t \to \infty} \rho(t)=\infty$.  With this fact,  the preceding  equivalence of $K$-functionals is precisely  the needed $K$-closedness. 
\end{proof}

As an  application  of Theorem~\ref{K-general}, we have the following general scheme of lifting  interpolation identities from a given   couple of  symmetric quasi-Banach  function spaces to  the corresponding couple of noncommutative conditioned column Hardy spaces. This  appears to be new even for classical martingale Hardy spaces.

\begin{corollary}\label{lifting}
Let $0<p< q<\infty$, $E \in {\rm Int}(L_p, L_q)$,  and 
$\cal{G}$ be a quasi-Banach function space with monotone quasi-norm.
 If $F= (E,L_\infty)_{\cal{G};K}$ and $F$ is $r$-concave for some $r<\infty$,   then 
\[
\h_F^c(\M)= (\h_E^c(\M),\h_\infty^c(\M))_{\cal{G};K}
\]
with equivalent quasi-norms.
\end{corollary}
\begin{proof}
From Proposition~\ref{K-method}, there is a quasi-Banach function space with monotone quasi-norm $\cal{F}$ so that $E=(L_p,L_\infty)_{\cal{F};K}$. Then by Theorem~\ref{K-general},  the couple 
$(\h_E^c(\M),\h_\infty^c(\M))$ is $K$-closed in  the couple $(E(\N), \N)$. On the other hand,  we have from the assumption and \eqref{comp-K} that   $F(\N) =(E(\N), \N)_{\cal{G}; K}$.  The conclusion follows immediately from $K$-closedness. 
\end{proof}

Next, we will point out that the ideas used  for  couples of noncommutative martingale Hardy spaces  can be adapted to   related  compatible couples.   Namely, we will consider  couples of conditioned $L_p$-spaces and 
couples of spaces of adapted sequences.

 We consider first the  couple of conditioned spaces  $(L_p^{\rm cond}(\M,\ell_2^c),L_\infty^{\rm cond}(\M,\ell_2^c))$.  We have a corresponding result to Proposition~\ref{main}.

\begin{lemma}\label{cond}
There exists a constant $C>0$ so that for every $a \in L_2^{\rm cond}(\M;\ell_2^c) + L_\infty^{\rm cond}(\M;\ell_2^c)$ and $t>0$, the following inequality holds:
\[
K\big(a, t; L_2^{\rm cond}(\M;\ell_2^c), L_\infty^{\rm cond}(\M;\ell_2^c) \big)\leq C \Big(\int_0^{t^2} \big(\mu_u(\sigma_c(a))\big)^2 \ du\Big)^{1/2}.
\]
\end{lemma}
This can be verified   by repeating the proof of Proposition~\ref{main} but using the sequence $(\sigma_{c,k}^2(a))_{k\geq 1}$ in place of $(s_{c,k}^2(x))_{k\geq 1}$ in the first step. We  omit   the details.

With  Lemma~\ref{cond} on hand,  we can repeat mutatis mutandis the  series of arguments   leading up to  the proof of Theorem~\ref{main-closed}  by using \cite[Proposition~3.7]{Ran-Int} in place of \cite[Corollary~3.8]{Ran-Int} and the fact  from \cite{Ju}  that  as in the case of Hardy spaces,   $L_p^{\rm cond}(\M)$ identifies as a complemented subspace of $L_p(\N)$ whenever $1<p<\infty$. We leave the details to the reader.

 We   state  below what we  consider the two most important results that we should retain from these adjustments. One is the $K$-closedness  result and the other is  the automatic liftings of interpolations.

\begin{proposition}\label{conditioned}
\begin{enumerate}[{\rm (i)}]
\item
For a given $0<p<\infty$, the    couple $\big(L_p^{\rm cond}(\M;\ell_2^c), L_\infty^{\rm cond}(\M;\ell_2^c) \big)$ is $K$-closed in  the  couple $\big(L_p(\N),\N \big)$.
\item
Let $0<p< q <\infty$,  $E \in {\rm Int}(L_p, L_q)$,  and 
$\cal{G}$ be a quasi-Banach function space with monotone quasi-norm.
 If $F= (E,L_\infty)_{\cal{G};K}$ and $F$ is $r$-concave  for some $r<\infty$,   then 
\[
F^{\rm cond}(\M;\ell_2^c)= (E^{\rm cond}(\M;\ell_2^c), L_\infty^{\rm cond}(\M;\ell_2^c))_{\cal{G};K}.
\]
\end{enumerate}
\end{proposition}

Similarly, by using square functions in place of conditioned square functions, the  proof  of Proposition~\ref{main} can be adjusted to prove  the following corresponding result for  couples of spaces of adapted sequences. 

\begin{lemma}\label{ad}
There exists a constant $C>0$ so that for every $a \in L_2^{\rm ad}(\M;\ell_2^c) +L_\infty^{\rm ad}(\M;\ell_2^c)$ and $t>0$, the following inequality holds:
\[
K\big(a, t; L_2^{\rm ad}(\M;\ell_2^c), L_\infty^{\rm ad}(\M;\ell_2^c) \big)\leq C \Big(\int_0^{t^2} \big(\mu_u({\cal S}_c(a))\big)^2 \ du\Big)^{1/2}.
\]
\end{lemma}
 
As in the case of conditioned spaces, we can rewrite all the arguments and results  from the  case of column conditioned Hardy spaces to  couples of spaces of adapted sequences.
Consequently, we obtain a much more general result  than \cite[Theorem~3.13]{Ran-Int} which we can state as follows:
\begin{proposition}\label{adapted}
\begin{enumerate}[{\rm (i)}]
\item
For a given $0<p<\infty$, the   couple $\big(L_p^{\rm ad}(\M;\ell_2^c), L_\infty^{\rm ad}(\M;\ell_2^c) \big)$ is $K$-closed in  the   couple $\big(L_p(\M;\ell_2^c), L_\infty(\M;\ell_2^c) \big)$.
\item
Let $0<p<\infty$,  $E \in {\rm Int}(L_p, L_\infty)$,  and 
$\cal{G}$ be a quasi-Banach function space with monotone quasi-norm.
 If $F= (E,L_\infty)_{\cal{G};K}$,   then 
\[
F^{\rm ad}(\M;\ell_2^c)= (E^{\rm ad}(\M;\ell_2^c), L_\infty^{\rm ad}(\M;\ell_2^c))_{\cal{G};K}.
\]
\end{enumerate}
\end{proposition}

As already known from \cite{Ran-Int}, the case  of adapted sequences implies  some corresponding results on Hardy spaces with natural restrictions on the indices. 
\begin{corollary}\label{K-H}
\begin{enumerate}[{\rm (i)}]
\item
If  $1<p<\infty$, then the couple $(\H_1^c(\M), \H_p^c(\M))$ is $K$-closed in  
the couple $(L_1(\M;\ell_2^c), L_p(\M;\ell_2^c))$.
\item Let $1<p<\infty$,  $E \in {\rm Int}(L_1, L_p)$,  and 
$\cal{G}$ be a quasi-Banach function space with monotone quasi-norm.
 If $F= (E,L_p)_{\cal{G};K}$,   then 
\[
\H_F^c(\M)= (\H_E^c(\M), \H_p^c(\M))_{\cal{G};K}.
\]
\end{enumerate}
\end{corollary}
\begin{proof} This follows from combining   the result on adapted sequence with  the noncommutative Stein inequality (\cite{PX}) and the noncommutative L\'epingle-Yor inequality (\cite{Qiu1}). Indeed, for $1\leq r<\infty$ we may identify $\H_r^c(\M)$ as a complemented subspace of $L_r^{\rm ad}(\M;\ell_2^c)$ via the projection $(a_n)_{n\geq 1} \mapsto \sum_{n\geq 1} a_n-\E_{n-1}(a_n)$. Thus, by complementation,  the couple $(\H_1^c(\M), \H_p^c(\M))$ is $K$-closed in   the couple 
 $(L_1^{\rm ad}(\M;\ell_2^c), L_p^{\rm ad}(\M;\ell_2^c))$. It follows  further  that it is $K$-closed in  the couple $(L_1(\M;\ell_2^c), L_p(\M;\ell_2^c))$. The second part also follows from complementation.
\end{proof}

Contrary to its conditioned counterpart,  the preceding corollary   only applies to the Banach space range. We refer to \cite{Ran-Int} for more details  on this fact.  It is also important to note that we do not have  the corresponding  result when the couple $(\H_1^c(\M), \H_\infty^c(\M))$  is used.

As an illustration of  the usefulness of  Corollary~\ref{K-H} and its row version,  we may easily deduce that one side of the Burkholder-Gundy inequality  (\cite[Corollary~4.3]{JX}) extends to a class of symmetric spaces of measurable operators.   More precisely, 
if $E \in {\rm Int}(L_1, L_2)$ then for every $x \in \H_E(\M)$, the following holds:
\begin{equation}\label{BG-1}
\|x\|_{E(\M)} \leq C_E \|x\|_{\H_E}.
\end{equation}
As  an example, consider the  Zygmund space   $L\log L$.  Recall that 
\[
L\log L=\big\{ f \in L_0 : \int_0^\infty |f(t)| \log^{+}|f(t)|\ dt <\infty \big\}.
\]
Equipped with the norm $\|f\|_{L\log L} =\int_0^\infty \mu_t(f) \log(1/t)\ dt$, $L\log L$ is a symmetric Banach space.
   It can be seen  from \cite{KaltonSMS} that  $L\log L \in {\rm Int}(L_1, L_2)$.  Therefore, \eqref{BG-1} applies to martingales in $L\log L(\M)$. More results in the spirit of \eqref{BG-1} will be explored in the appendix section below.


\subsection{ Some concrete illustrations}

Below, we consider   two concrete   examples on how our general result from the previous two  subsections can be applied.

 \subsubsection{Noncommutative  Orlicz spaces}
 We will first review the basics of Orlicz spaces.

  A function $\Phi: [0,\infty) \to [0,\infty)$ is  called an \emph{Orlicz  function} whenever it is strictly increasing, continuous,  $\Phi(0)=0$, and $\lim_{u\to \infty}\Phi(u)=\infty$.
 The \emph{Orlicz space} $L_\Phi$ is the collection of all  $f\in L_0$ for which 
   there exists a constant $c$  such that   $I_\Phi(|f|/c)<\infty$  where the modular functional $I_\Phi(\cdot)$ is defined by:
 \[
 I_\Phi(|g|)=\int_0^\infty \Phi(|g(t)|)\ dt, \quad g \in L_0.
 \]
 
 We recall that for $0<p \leq q<\infty$, $\Phi$ is called  \emph{$p$-convex} (resp., \emph{$q$-concave}) if  the function $t\mapsto \Phi(t^{1/p})$ (resp., $t\mapsto \Phi(t^{1/q})$) is convex (resp., concave).
Below, we only consider  Orlicz spaces associated  with Orlicz  functions  that are  $p$-convex and $q$-concave for some $0<p\leq q< \infty$.
It is well-known that  $L_\Phi$ is a  linear space.  We set:
\[
\big\|f \big\|_{\Phi}= \inf\big\{ c>0 : I_\Phi(|f|/c) \leq 1\big\}.
\] 
If $\Phi$ is convex,  then  $\|\cdot\|_\Phi$ is a norm for which $(L_\Phi, \|\cdot\|_\Phi)$ is a symmetric Banach function space. However,  when $0<p<1$, then  $\|\cdot\|_\Phi$ is only a quasi-norm for which  $(L_\Phi, \|\cdot\|_\Phi)$ is  a symmetric quasi-Banach function space. Our reference for Orlicz functions and Orlicz spaces is the monograph \cite{Maligranda2}.

 Below,  we will also make use of the  following space:  for $0<r\leq \infty$, the space $L_{\Phi,q}$ is the collection of all $f \in L_0$ for which 
$\| f\|_{\Phi,r} <\infty$ where 
\begin{equation*}
\big\|f \big\|_{\Phi,r} :=\begin{cases}
 \left(r\displaystyle{\int_{0}^\infty \big(t \| \ch_{\{|f|>t\}} \|_\Phi}\big)^r\ \frac{dt}{t}\right)^{1/r},  &0< r < \infty; \\
\displaystyle{\sup_{t >0} t \| \ch_{\{|f|>t\}} \|_\Phi}\, &r= \infty.
\end{cases} 
 \end{equation*}
The space $L_{\Phi,r}$ was introduced in \cite{Hao-Li} and was called Orlicz-Lorentz
there. We should warn the reader that  this is different from  Orlicz-Lorentz spaces  used elsewhere  in the literature  such as \cite{SMS-Orlicz}. Note that if $\Phi(t)=t^p$, then  $L_{\Phi,r}$ coincides with  the Lorentz space $L_{p,r}$. The space $L_{\Phi,\infty}$ is also known as the weak Orlicz space. Below, we use  the notation $\h_{\Phi}^c(\M)$ and $\h_{\Phi,r}^c(\M)$ for  noncommutative column conditioned Hardy spaces associated with $L_\Phi$ and $L_{\Phi,r}$ respectively.

\smallskip

We begin with the following statement  at the level of function spaces:
\begin{proposition}[{\cite[Proposition~3.3]{L-T-Zhou}}]
Let $\Phi$ be an Orlicz function, $0<\g\leq \infty$, and  $0<\theta<1$. If  $\Phi_0^{-1}(t) =[\Phi^{-1}(t)]^{1-\theta}$, then 
\[
(L_\Phi, L_\infty)_{\theta, \g} =L_{\Phi_0,\g}.
\]
\end{proposition}

We can deduce   immediately the next interpolation result  from the preceding proposition  and Corollary~\ref{lifting}. 
 \begin{theorem}\label{Orlicz}
 Let $0<\theta <1$ and  $0< \g \leq \infty$. If $\Phi$ is an Orlicz function that is $p$-convex and $q$-concave for $0<p<q<\infty$,  then   for $\Phi_0^{-1}(t) =[\Phi^{-1}(t)]^{1-\theta}$, the following  holds:
 \[
\big( \h_{\Phi}^c(\M), \h_\infty^c(\M)\big)_{\theta, \g} =\h_{\Phi_0,\g}^c(\M).
\]
\end{theorem}

We note  that 
by reiteration, we  obtain that  for $0<\lambda\leq \infty$, the following also holds:
\[
\big( \h_{\Phi,\lambda}^c(\M), \h_\infty^c(\M)\big)_{\theta, \g} =\h_{\Phi_0,\g}^c(\M).
\] 
Moreover,   let  $0<\theta, \eta<1$, and $0<\lambda,  \gamma \leq \infty$. Set   $\Psi_1$ and $\Psi_2$ 
such that $\Psi_1^{-1}(t) =[\Phi^{-1}(t)]^{1-\theta}$ and $\Psi_2^{-1}(t) =[\Phi^{-1}(t)]^{1-\theta\eta}$.  Then, we also have:
\begin{equation}\label{iteration}
\big( \h_{\Phi}^c(\M), \h_{\Psi_1, \lambda}^c(\M)\big)_{\eta, \g} =\h_{\Psi_2,\g}^c(\M).
\end{equation}

Some partial  results for the case of  classical martingale Hardy spaces were  obtained recently  in \cite{ L-T-Zhou}.  We also refer to  the recent article  \cite{L-Weisz-Xie} for results of similar nature  on classical martingale Hardy spaces.

\medskip

From the discussions above, we have  the corresponding results for conditioned spaces and spaces of adapted sequences.

\begin{proposition}
Under the assumptions of the previous theorem, the following hold:
\begin{enumerate}[{\rm (i)}]
\item 
 $\displaystyle{
\big( L_{\Phi}^{\rm cond}(\M;\ell_2^c), L_\infty^{\rm cond}(\M; \ell_2^c)\big)_{\theta, \g} =L_{\Phi_0,\g}^{\rm cond}(\M;\ell_2^c)}$;
\item
 $\displaystyle{
\big( L_{\Phi}^{\rm ad}(\M;\ell_2^c), L_\infty^{\rm ad}(\M; \ell_2^c)\big)_{\theta, \g} =L_{\Phi_0,\g}^{\rm ad}(\M;\ell_2^c)}$. 
\end{enumerate}
 \end{proposition}

We take the opportunity to present  interpolation   for martingale Orlicz Hardy space and martingale \lq\lq little\rq\rq \, BMO-space. We refer to \cite{Bekjan-Chen-Perrin-Y, Ran-Int} for definition of $\bmo^c(\M)$. The following result is a noncommutative analogue of \cite[Theorem~4.1]{L-Weisz-Xie}.
\begin{theorem}\label{Orlicz-bmo}
Let $\Phi$ be an Orlicz function that is $p$-convex and $q$-concave for $0<p\leq q<\infty$. If $0< \theta<1$ and $1<\g \leq\infty$, then  for $\Phi_0^{-1}(t) =[\Phi^{-1}(t)]^{1-\theta}$, the following  holds:
\[
\big( \h_{\Phi}^c(\M), \bmo^c(\M) \big)_{\theta, \g} =\h_{\Phi_0,\g}^c(\M).
\]
\end{theorem}

In the classical setting, the standard  procedure for deducing Theorem~\ref{Orlicz-bmo} from Theorem~\ref{Orlicz} is  by evaluating the norm  of sharp functions (see \cite{Weisz}). Our argument below  is quite different. It uses ideas from  \cite{Bekjan-Chen-Perrin-Y} in this context of Orlicz spaces. Indeed,
our  proof  for the Banach space range combines duality argument with $K$-closedness in the spirit of Theorem~\ref{K-general}. The full generality is deduced using Wolff interpolation theorem.

 We will make use of the next lemma which  can be easily deduced from the complementation result stated in Proposition~\ref{complemented}. 
\begin{lemma}\label{duality}
Assume that $G \in {\rm Int}(L_r, L_s)$ for some $1<r\leq s<\infty$. If $G$ is separable, then 
\[
(\h_G^c(\M))^* =\h_{G^*}^c(\M).
\]
\end{lemma}
\begin{proof}[Proof of Theorem~\ref{Orlicz-bmo}]
We divide the proof into several cases.

$\bullet$ Case~1:  Assume that $1<p \leq q<\infty$. Denote by $\Phi^*$  the Orlicz function complementary to  the convex function $\Phi$. Then $\Phi^*$ is $q'$-convex and $p'$-concave where $p'$ and $q'$ denote the conjugate indices of $p$ and $q$ respectively.  It is known that $L_{\Phi^*} \in {\rm Int}(L_{q'}, L_{p'})$. A fortiori, $L_{\Phi^*} \in {\rm Int}(L_{1}, L_{p'})$. Let $\cal{F}$ be a Banach function space with monotone norm so that $L_{\Phi^*}=(L_{p'}, L_1)_{\cal{F}; K}$.  The existence of such $\cal{F}$ is given   by Proposition~\ref{K-method}. By Proposition~\ref{gen-Holm}, one can express the $K$-functionals of the couple $(L_{\Phi^*}, L_1)$ in terms of those in $(L_{p'}, L_1)$.
 Applying  similar argument as in the proof of  Theorem~\ref{K-general}, we
can deduce that  the couple $(\h_{\Phi^*}^c(\M), \h_1^c(\M))$ is $K$-closed in  the couple $(L_{\Phi^*}(\N), L_1(\N))$. 
Therefore, if  $\big(L_{\Phi^*}, L_1\big)_{\theta, \gamma'}=F$ then 
\[
\big(\h_{\Phi^*}^c(\M), \h_1^c(\M)\big)_{\theta, \gamma'} =\h_F^c(\M)
\]
where $\gamma'$ is  the conjugate index of $\gamma$.   Under the assumption on $\Phi$, we have  from Lemma~\ref{duality} that $(\h_{\Phi^*}^c(\M))^* = \h_\Phi^c(\M)$ and 
$(\h_F^c(\M))^*=\h_{F^*}^c(\M)$. Since $1\leq \gamma'<\infty$ and $\h_{\Phi^*}^c(\M) \cap \h_1^c(\M)$ is dense in both $\h_{\Phi^*}^c(\M)$ and $\h_1^c(\M)$, we can apply    the duality  theorem for real interpolation (see \cite[Theorem~3.7.1]{BL}).   With  the    fact that $(\h_1^c(\M))^*=\bmo^c(\M)$, it implies:
\[
\big(\h_{\Phi}^c(\M), \bmo^c(\M)\big)_{\theta, \gamma} =\h_{F^*}^c(\M).
\]
Next, we observe that the duality for interpolation  also gives $F^*=(L_\Phi,  L_\infty)_{\theta, \gamma}=L_{\Phi_0, \gamma}$. This completes the proof for this case. Below, we will use the 
following form which  can be easily deduced  by reiteration: if $1<p\leq q<\infty$ then for every $1 <\lambda \leq \infty$,
\begin{equation}\label{Phi-bmo-1}
\big(\h_{\Phi,\lambda}^c(\M), \bmo^c(\M)\big)_{\theta, \gamma} =\h_{\Phi_0, \gamma}^c(\M).
\end{equation}

$\bullet$  Case~2.  Assume that $p(1-\theta)^{-1} >1$. One can easily see that $\Phi_0$ is $p(1-\theta)^{-1}$-convex and $q(1-\theta)^{-1}$-concave.
Fix $1-p<\psi<\theta$ and define $\Phi_1$ so  that 
\[
\Phi_1^{-1}(t) =[\Phi^{-1}(t)]^{1-\psi}, \quad t>0.
\]
We note that $\Phi_1$ is $p(1-\psi)^{-1}$-convex  with $p(1-\psi)^{-1}>1$ and $\Phi_0^{-1}(t)=[ \Phi_1^{-1}(t)]^{1-\theta_0}$ for $\displaystyle{\theta_0=1-\frac{1-\theta}{1-\psi}}$.  Applying  \eqref{Phi-bmo-1} with $\Phi_1$, we have 
\begin{equation*}
\big(\h_{\Phi_1,\lambda}^c(\M), \bmo^c(\M)\big)_{\theta_0, \gamma} =\h_{\Phi_0,\gamma}^c(\M).
\end{equation*}
On the other hand, we also have from \eqref{iteration} that
\begin{equation*}
\big(\h_{\Phi}^c(\M), \h_{\Phi_0, \gamma}^c(\M)\big)_{\theta_1, \lambda} =\h_{\Phi_1,\lambda}^c(\M)
\end{equation*}
where $\theta_1= \psi/\theta$.  By Wolff's interpolation theorem, it follows that 
\begin{equation*}
\big(\h_{\Phi}^c(\M), \bmo^c(\M)\big)_{\xi, \lambda} =\h_{\Phi_0,\lambda}^c(\M)
\end{equation*}
where $\displaystyle{\xi=\frac{\theta_0}{1-\theta_1 +\theta_1\theta_0}}$. A simple calculation shows that $\xi=\theta$.

$\bullet$ Case~3.  Assume that $p(1-\theta)^{-1} \leq 1$. Set $\widetilde{\Phi}$ so that for $t>0$,  $\widetilde{\Phi}^{-1}(t) = [\Phi_0^{-1}(t)]^p$. Then $p_0 +\psi>1$ where $p_0=p(1-\theta)^{-1}$ and $\psi=1-p$. Using Case~2 with $\Phi_0$ in place of $\Phi$ and $\widetilde{\Phi}$ in place of $\Phi_0$, we get
\begin{equation*}
\big(\h_{\Phi_0,\lambda}^c(\M), \bmo^c(\M)\big)_{\psi, \lambda} =\h_{\widetilde{\Phi},\lambda}^c(\M).
\end{equation*}
Next, we  note that  for $t>0$,  $\widetilde{\Phi}(t)^{-1}= [ \Phi^{-1}(t)]^{p(1-\theta)}$. By applying \eqref{iteration}, we get that
\begin{equation*}
\big(\h_{\Phi}^c(\M), \h_{\widetilde{\Phi},\lambda}^c(\M)\big)_{\eta, \lambda} =\h_{\Phi_0,\lambda}^c(\M)
\end{equation*}
where $\displaystyle{\eta= \frac{\theta}{1-p(1-\theta)}}$. It follows from Wolff's interpolation theorem that
\begin{equation*}
\big(\h_{\Phi}^c(\M), \bmo^c(\M)\big)_{\upsilon, \lambda} =\h_{\Phi_0,\lambda}^c(\M)
\end{equation*}
where 
$\displaystyle{\upsilon= \frac{\eta\psi}{1-\eta-\eta\psi}}$. One can easily verify  that $\upsilon=\theta$. The proof is complete.
\end{proof}


\subsubsection{Generalized Lorentz spaces}

We now  examine  the  class of  generalized Lorentz  spaces. Our motivation for considering such class comes from \cite{Ren-Guo} where classical martingale Hardy spaces associated with  generalized Lorentz spaces were studied.

Let $\varphi:(0,\infty) \to (0,\infty)$  be a  locally integrable   function. Given $0<r \leq \infty$, the  \emph{Lorentz space}  $\Lambda^{r}(\varphi)$ is the linear subspace of  all function $ f\in L_0$  for which $\|\cdot\|_{\Lambda^r(\varphi)}<\infty$ where
\begin{equation*}
\big\| f\big\|_{\Lambda^r(\varphi)}:=\begin{cases}
\displaystyle{\Big( \int_0^\infty (\mu_t(f) \varphi(t))^r\  \frac{dt}{t} \Big)^{1/r} },  &0< r < \infty; \\
\displaystyle{\sup_{t >0}  \mu_t(f)\varphi(t)}, &r= \infty.
\end{cases} 
\end{equation*}

The space  $(\Lambda^{r}(\varphi), \|\cdot\|_{\Lambda^r(\varphi)})$ is  a symmetric quasi-Banach function space. We should note that there are other equivalent formulations of the generalized Lorentz spaces but we chose to follow the presentation of \cite{Persson} since many properties we use below are taken  directly from \cite{Persson}.

 Let $0<a_1<a_2$. Following \cite{Persson}, we say that a function $\varphi $ belongs to the class $Q[a_1,a_2]$ if $t\mapsto  t^{-a_1}\varphi(t)$ is nondecreasing and $t\mapsto  t^{-a_2}\varphi(t)$ is nonincreasing.

For simplicity,  we denote by $\h_{r,\varphi}^c(\M)$ the  noncommutative column conditioned 
Hardy space associated with the Lorentz space $\Lambda_{r}(\varphi)$.

We will now describe an interpolation method that is suitable for Lorentz spaces. For a given function $\varrho \in Q(0,1)$ and $0<q \leq  \infty$, consider the function space $\cal{F}_{\varrho,q}$ defined by  the set of all functions $f\in L_0$  satisfying $\|f\|_{\cal{F}_{\varrho,q} }<\infty$  where
\begin{equation*}
\big\| f\big\|_{\cal{F}_{\varrho,q}}:=\begin{cases}
\displaystyle{\Big( \int_0^\infty (|f(t)|/\varrho(t))^q\  \frac{dt}{t} \Big)^{1/q} },  &0< q < \infty; \\
\displaystyle{\sup_{t >0}  |f(t)|/\varrho(t)}, &q= \infty.
\end{cases} 
\end{equation*}
Clearly, $\cal{F}_{\varrho,q}$ equipped with the above quasi-norm is a quasi-Banach function space with monotone quasi-norm.  Following \cite{Persson}, for a compatible couple $(A_0,A_1)$,  we denote by $(A_0,A_1)_{\varrho,q}$ the interpolation space given by $(A_0, A_1)_{\cal{F}_{\varrho,q}; K}$. If $\varrho(t)=t^\theta$ for $0<\theta<1$, then $(A_0,A_1)_{\varrho,q}$ reduces to $(A_0,A_1)_{\theta,q}$. 

Our starting point is the following   interpolation involving Lorentz spaces.

\begin{proposition}[{\cite[Proposition~6.2]{Persson}}]
If $\varphi \in Q(0,b)$ for some $0<b<\infty$, then
\[
\big( \Lambda^{p}(\varphi) , L_\infty\big)_{\varrho,q} = \Lambda^{q}(\varphi_0)
\]
where $\varphi_0(t) =\varphi(t)/\varrho(\varphi(t))$.
\end{proposition}

As an immediate consequence of Corollary~\ref{lifting}, we  have:

\begin{theorem}\label{Lorentz-Hardy}
If $\varphi \in Q(0,b)$ for some $0<b<\infty$, then
\[
\big( \h_{p, \varphi}^c(\M), \h_\infty^c(\M)\big)_{\varrho,q} = \h_{q, \varphi_0}^c(\M)
\]
where $\varphi_0(t) =\varphi(t)/\varrho(\varphi(t))$.
\end{theorem}

Next, we will show as in the case of Orlicz spaces that the same result holds if  we use $\bmo^c(\M)$ as one the endpoints in the interpolation in place of $\h_\infty^c(\M)$.  We begin with the following intermediate lemma.

\begin{lemma}\label{h-bmo-Lorentz}
Let $0<p<\infty$, $1<q\leq \infty$, and $\varrho \in Q(0,1)$. Then
\[
\big(\h_p^c(\M), \bmo^c(\M) \big)_{\varrho, q} = \h_{q, \varphi}^c(\M)
\]
where $\varphi(t)=t^{1/p}/ \varrho(t^{1/p})$.
\end{lemma}
\begin{proof}
We will only outline the proof. Assume first that $1<p<\infty$. By the duality theorem stated in 
 \cite[Theorem~2.4]{Persson}, we have
\[
\big(\h_p^c(\M), \bmo^c(\M) \big)_{\varrho, q} =(\big(\h_{p'}^c(\M), \h_1^c(\M) \big)_{\varrho_1, q'} )^*
\]
where $\varrho_1(t)= 1/ \varrho(1/t)$ and $p'$ and $q'$ are the index conjugate of $p$ and $q$ respectively. By $K$-closedness,  if $(L_{p'},L_1)_{\varrho_1,q'}=F$ then 
$\big(\h_{p'}^c(\M), \h_1^c(\M) \big)_{\varrho_1, q'}=\h_F^c(\M)$. Thus, it suffices to understand the space $F^*= (L_p, L_\infty)_{\varrho,q}$ which by \cite[Lemma~6.1]{Persson} is the space $\Lambda^q(t^{1/p}/\varrho(t^{1/p}))$. This proves   the case $1<p<\infty$.

For the case $0<p\leq 1$, we may apply the general form of Wolff interpolation theorem proved in \cite[Theorem~5.3]{Persson}. We  omit the details.
\end{proof}

The next result is the BMO-version of  Theorem~\ref{Lorentz-Hardy}. It is the noncommutative analogue of \cite[Theorrm~4.1]{Ren-Guo}.
\begin{theorem}
Let $\varphi_0(t)$ and $\varrho(t)$ be two functions in $Q(0,1)$,  $1< q_0, q<\infty$. Then
\[
\big( \h_{q_0, \varphi}^c(\M), \bmo^c(\M)\big)_{\varrho,q} = \h_{q, \varphi_0}^c(\M)
\]
where $\varphi_0(t) =\varphi(t)/\varrho(\varphi(t))$.
\end{theorem}
\begin{proof}
Let $\varrho_0(t)=t/\varphi(t^p)$.  Equivalently, $\varphi(u)=u^{1/p}/\varrho_0(u^{1/p})$. Choose $p$ small enough  so that $\varrho_0 \in Q(0,1)$. Applying Lemma~\ref{h-bmo-Lorentz},  we have
\begin{align*}
\big( \h_{q_0, \varphi}^c(\M), \bmo^c(\M)\big)_{\varrho,q} &=\big( (\h_p^c(\M), \bmo^c(\M))_{\varrho_0, q_0}, \bmo^c(\M)\big)_{\varrho,q} \\
&=\big(\h_p^c(\M), \bmo^c(\M)\big)_{\varrho_0(t)\varrho(t/\varrho_0(t)), q}\\
\end{align*}
where the second identity comes from  the reiteration formula given  in \cite[Theorem~1.1]{Persson}. Reapplying Lemma~\ref{h-bmo-Lorentz}, we conclude that
\[
\big( \h_{q_0, \varphi}^c(\M), \bmo^c(\M)\big)_{\varrho,q}=  \h_{q, \varphi_0}^c(\M)
\]
with $\varphi_0$ as claimed.
\end{proof}

\appendix
\section{Martingale inequalities}

In this  appendix section, we apply results from the previous section to martingale inequalities. 
 Our purpose is to point out that  new developments made   from previous sections  lead  to  improvements to  all  results from \cite[Section~4]{Ran-Int}. In particular, we answered a problem left open in \cite{Ran-Wu-Xu}.
 
  Before we proceed, 
 we record the following result  for further use. We may view this as  a $\Phi$-moment analogue of the $K$-monotonicity. We refer to \cite{Bekjan-Chen, RW2}  for more information on $\Phi$-moment inequalities.

\begin{proposition}\label{Phi-inter}
Let $0<p<q<\infty$ and $\Phi$ be  a  $p$-convex and $q$-concave Orlicz function.
Assume that $f$ and $g$ are positive functions in $L_0$ such that $g \in L_\Phi$  and 
for every $t>0$,  the inequality $K(f,t;L_p,L_q) \leq  K(g,t;L_p,L_q)$ holds. Then 
\[
\int_0^\infty \Phi(f(t)) \ dt \lesssim_{p,q} \int_0^\infty \Phi(g(t))\ dt.
\]
\end{proposition}
\begin{proof}
Assume first that $\Phi_0$ is an  Orlicz function that is $p_0$-convex and $q_0$-concave for  some $1\leq p_0<q_0<\infty$. Let  $f_0$ and $g_0$ be  positive functions in $L_0$ such that  $g_0 \in L_{\Phi_0}$ and 
for every $t>0$,  the inequality $K(f_0,t;L_{p_0},L_{q_0}) \leq  K(g_0,t;L_{p_0},L_{q_0})$ holds. 

According to \cite{Sparr}, the pair  $(L_{p_0}, L_{q_0})$ is  a Calder\'on couple. That is, there exists an operator $T: L_{p_0} + L_{q_0} \to L_{p_0} + L_{q_0}$ with $T(L_{p_0}) \subseteq L_{p_0}$,  $T(L_{q_0}) \subseteq L_{q_0}$, and such that $Tg_0=f_0$. 
A closer inspection  of \cite{Sparr}  reveals  that $\max\{ \|T: L_{p_0} \to L_{p_0}\|, 
\|T: L_{p_0} \to L_{p_0}\|\} \lesssim_{p_0, q_0} 1$. We  can  now appeal to \cite[Lemma~6.2]{Jiao-Sukochev-Zanin} (see also \cite[Lemma~3.10]{Ran-Wu-Zhou} for a remark about  the constant) to conclude that 
\[
\int_0^\infty \Phi_0(f_0(t)) \ dt \lesssim_{p_0,q_0} \int_0^\infty \Phi_0(g_0(t))\ dt.
\]
Therefore, the proposition is verified for convex Orlicz functions.

Assume now that $0<p<1$ and $p<q<\infty$. Consider an Orlicz function $\Phi$  that is $p$-convex and $q$-concave. Define $\Phi_0(t)=\Phi(t^{1/p})$.  Then $\Phi_0$ is convex and $q/p$-concave. One can easily see from Holmsted's formula \cite[Theorem~4.1]{Holm} that for a given positive function $w$ and $t>0$, $K(w, t^{1/p}; L_p, L_q) \approx_{p,q} [K(w^p, t; L_1, L_{q/p})]^{1/p}$.  Thus,  if $f$ and $g$ are two functions satisfying the 
 assumption of the proposition, then  we have  for every $t>0$,
$K(f^p,t;L_1,L_{q/p}) \lesssim_{p,q}  K(g^p,t;L_1,L_{q/p})$. We can apply  the previous  convex case  to deduce that 
\[
\int_0^\infty \Phi_0(f(t)^p) \ dt \lesssim_{p,q} \int_0^\infty \Phi_0(g(t)^p)\ dt,
\]
which is precisely  the desired conclusion.
\end{proof}

Our first result is a  strengthening of  \cite[Theorem~4.5]{Ran-Int}. It allows the inclusion of  $L_2$ as one of the endpoints in the interpolation.

\begin{theorem}\label{reverse-cond} 
Let  $0<p< 2$ and $F\in  {\rm Int}(L_p, L_2)$. There exists a constant $C_F$  such that   for any  $x \in F^{\rm cond}(\M;\ell_2^c)$, the following holds:
\[
\big\| x \big\|_{F(\M;\ell_2^c)}\leq C_F \big\| x \big\|_{F^{\rm cond}(\M;\ell_2^c)}.
\]
Similarly, if $\Phi$ is an Orlicz function that is $p$-convex and $2$-concave for $0<p<2$, then there exists a constant $C_{p}$ so that for any sequence $x=(x_k)_{k\geq 1}$ with $\sigma_c(x) \in L_\Phi(\M)$,
\[
\T\big[\Phi\big( \cal{S}_c(x) \big)\big] \leq C_{p}\T\big[\Phi\big(\sigma_c(x)\big)\big].
\]
\end{theorem}
\begin{proof}
First,  we note that  the  inequality holds  for $L_r(\M)$ where $0<r \leq 2$.  The case $r=2$ is trivial. If $0<r<2$ and  $x=(x_n)_{n\geq 1} \in  \mathfrak{F}$, then
\begin{align*}
\|x\|_{L_r(\M;\ell_2^r)}^r &= \|\sum_{n\geq 1} |x_k|^2 \|_{r/2}^{r/2}\\
&\leq  4\|\sum_{n\geq 1} \E_{k-1}|x_k|^2 \|_{r/2}^{r/2}\\
&=4\|x \|_{L_r^{\rm cond}(\M;\ell_2^c)}^r
\end{align*}
where the inequality in the second line comes from \cite[Theorem~7.1]{JX}.  By interpolation, the preceding  inequality  lifts  to $F\in  {\rm Int}(L_p, L_2)$ for $0<p<2$.

For the case of the  $\Phi$-moment, we start from stating  that for  $x \in \mathfrak{F}$, we have 
\begin{equation}\label{K}
K(x,t ; L_p(\M; \ell_2^c), L_2(\M;\ell_2^c)) \lesssim_p K(x,t; L_p^{\rm cond}(\M;\ell_2^c), L_2^{\rm cond}(\M;\ell_2^c)), \quad t>0.
\end{equation}
This can be deduced as follows: fix   $x \in  L_p^{\rm cond}(\M;\ell_2^c) +L_2^{\rm cond}(\M;\ell_2^c)$ and $t>0$. For $\epsilon>0$, choose $x=y +z$, with $y \in L_p^{\rm cond}(\M;\ell_2^c)$ and $z \in L_2^{\rm cond}(\M;\ell_2^c)$  and such that 
\[
\|y\|_{L_p^{\rm cond}(\M;\ell_2^c)} +\|z\|_{L_2^{\rm cond}(\M;\ell_2^c)} \leq K(x,t; L_p^{\rm cond}(\M;\ell_2^c), L_2^{\rm cond}(\M;\ell_2^c)) +\epsilon.
\]
Using the fact that the above inequality holds for $L_r$ for all $0<r\leq 2$, we have
\begin{align*}
K(x,t ; L_p(\M; \ell_2^c), L_2(\M;\ell_2^c))  &\leq \|y\|_{L_p^(\M;\ell_2^c)} + t\|z\|_{L_2(\M;\ell_2^c)} \\
&\leq C_p \|y\|_{L_p^{\rm cond}(\M;\ell_2^c)} + t\|z\|_{L_2^{\rm cond}(\M;\ell_2^c)} \\
&\leq \max(C_p, 1)( K(x,t; L_p^{\rm cond}(\M;\ell_2^c), L_2^{\rm cond}(\M;\ell_2^c)) +\epsilon).
\end{align*}
Since $\epsilon$ is arbitrary, we have \eqref{K}.
 
Next, it follows from Proposition~\ref{conditioned} that  \eqref{K}  is equivalent to
\[ 
K(\cal{S}_c(x),t ; L_p(\M), L_2(\M)) \lesssim_p K(\sigma_c(x),t ; L_p(\M), L_2(\M)).
\]
On the other hand,  according to \cite{PX3}(see Corollary~2.3 and the remarks afterward), this is  further equivalent to 
\[ 
K(\mu(\cal{S}_c(x)),t ; L_p, L_2) \lesssim_p K(\mu(\sigma_c(x)),t ; L_p, L_2).
\]
We can deduce from Proposition~\ref{Phi-inter} that
\[
\int_0^\infty \Phi(\mu_t(\cal{S}_c(x))) \ dt \lesssim_p \int_0^\infty \Phi(\mu_t(\sigma_c(x))) \ dt
\] 
and this is precisely the stated $\Phi$-moment.
\end{proof}

In turn,  Theorem~\ref{reverse-cond}  implies the following improvement  of \cite[Theorem~7.1]{JX} to  symmetric spaces of measurable operators. The argument  is identical to the proof  of \cite[Corollary~4.6]{Ran-Int} so we omit the details.

\begin{corollary}\label{reverse-dual-doob} Let $E$ be a  symmetric quasi-Banach  function space with $E\in {\rm Int}(L_p,L_1)$ for  some $0<p< 1$. There exists a constant $C_E$ so that for any sequence of positive operators $(a_k)$  in $\mathfrak{F}$, the following holds:
 \[
 \Big\| \sum_{k\geq 1} a_k \Big\|_{E(\M)} \leq C_E \Big\| \sum_{k\geq 1} \E_k(a_k)\Big\|_{E(\M)}.
 \] 
 Similarly,  if $\Phi$  is a  concave Orlicz function that is $p$-convex for  $0<p<1$,  then there exists a constant $C_{p}$ so that for any  sequence of positive operators $(a_k)$  in $\mathfrak{F}$, the following holds:
 \[
\T\big[ \Phi\big(\sum_{k\geq 1} a_k\big) \big]  \leq C_{p} \T\big[\Phi\big( \sum_{k\geq 1} \E_k(a_k)\big)\big].
 \]
 \end{corollary}
 
 We take the opportunity to present below a simple approach to the dual Doob inequality in the spirit  of the approach to Corollary~\ref{reverse-dual-doob}. The result below strengthens \cite[Corollary~4.13]{Dirksen2}.
 \begin{proposition}\label{DD}
Let $E$ be a symmetric Banach function space with $E \in {\rm Int}(L_1, L_q)$ for some $1<q<\infty$. There exists  a constant  $C_E$ so that for any sequence of  positive operators $(x_k)$ in $E(\M)$, the following holds:
\begin{equation}\label{DD-sym}
 \Big\| \sum_{k\geq 1} \E_k(x_k) \Big\|_{E(\M)} \leq C_E \Big\| \sum_{k\geq 1} x_k\Big\|_{E(\M)}.
 \end{equation}
 Similarly,  if $\Phi$  is a  convex Orlicz function that is $q$-concave  for  some  $1<q<\infty$,  then there exists a constant $C_{q}$ so that for any  sequence of positive operators $(x_k)$  in $L_\Phi(\M)$, the following holds:
 \[
\T\big[ \Phi\big(\sum_{k\geq 1} \E_k(x_k)\big) \big]  \leq C_{q} \T\big[\Phi\big( \sum_{k\geq 1} x_k\big)\big].
 \]
 \end{proposition}
 \begin{proof}
 We  only present the $\Phi$-moment case.  For $2\leq p<\infty$, we set $L_p^{\rm cond_+}(\M;\ell_2^c)$ to be the completion of the set of finite sequences in $L_p(\M) \cap \M$ under the norm
 \[
 \big\| (a_k)_{k\geq 1}\big\|_{L_p^{\rm cond_+}(\ell_2^c)}=\big\| \big(\sum_{k\geq 1} \E_k(|a_k|^2) \big)^{1/2} \big\|_p. 
 \]
This is  a slight modification of  the conditioned $L_p$-space $L_p^{\rm cond}(\M;\ell_2^c)$  using $(\E_k)_{k\geq 1}$ in place of $(\E_{k-1})_{k\geq 1}$. Using this variant, we still  have that $L_p^{\rm cond_+}(\M;\ell_2^c)$ embeds isometrically into $L_p(\M \overline{\otimes} B(\ell_2(\mathbb N^2)))$. We denote by $U^+$ such isometry.
 
 We now proceed with the proof. We start with  the simple observation that for every 
 $2\leq p<\infty$, the identity map $I: L_p(\M;\ell_2^c) \to L_p^{\rm cond_+}(\M;\ell_2^c)$
 is bounded. This is equivalent to the noncommutative dual Doob inequality from \cite{Ju} for the index $p/2$. Denote by $\Pi$ the natural contractive projection  from $L_p(\M\overline{\otimes} \cal{B}(\ell_2(\mathbb{N})))$ onto  $L_p(\M;\ell_2^c)$. Then 
  $U^+ I  \Pi: L_p(\M\overline{\otimes} \cal{B}(\ell_2(\mathbb{N}))) \to L_p(\M\overline{\otimes} \cal{B}(\ell_2(\mathbb{N}^2)))$ is bounded with norm depending only on $p$.

If $\varphi$ is an Orlicz function that is $2$-convex and $2q$-concave  then  for every $\xi \in L_\varphi (\M\overline{\otimes} \cal{B}(\ell_2(\mathbb{N})))$,
\[
\T\otimes\Tr\big[\varphi\big(|U^+ I \Pi (\xi)|\big) \big] \lesssim_q \T\otimes\tr\big[\varphi\big(|\xi|\big) \big] 
\]
where $\tr$ and $\Tr$ denote the usual trace on $\cal{B}(\ell_2(\mathbb{N}))$  and  $\cal{B}(\ell_2(\mathbb{N}^2))$ respectively. Let $a=(a_k)$ be a sequence in $L_\varphi(\M;\ell_2^c)$. Using $\xi=\sum_k a_k \otimes e_{k,1}$, the above inequality becomes
\begin{equation}\label{varphi}
\T\big[ \varphi\big(\sigma_c^+(a)\big)\big] \lesssim_q \T\big[ \varphi\big(\cal{S}_c(a)\big)\big]
\end{equation}
where $\sigma_c^+(a)=\big(\sum_{k\geq 1} \E_k(|a_k|^2)\big)^{1/2}$.

To conclude the proof, let $\Phi$ as in the statement and $(x_k)_k$ be  a sequence of positive operators  from  $L_\Phi(\M)$. For $t\geq0$, set $\varphi(t)=\Phi(t^2)$. Then $\varphi$ is  a $2$-convex and $2q$-concave Orlicz function. Consider the sequence  $a=(x_k^{1/2})_{k\geq 1}$. It is easy to see that $\cal{S}_c(a) =\big(\sum_{k\geq 1} x_k\big)^{1/2}$ and $\sigma_c^+(a)=\big( \sum_{k\geq 1} \E_k(x_k)\big)^{1/2}$. It follows from \eqref{varphi} that
\[
\T\big[ \Phi\big(\sum_{k\geq 1} \E_k(x_k)\big) \big] =\T\big[ \varphi\big(\sigma_c^+(a)\big)\big] \lesssim_q \T\big[ \varphi\big(\cal{S}_c(a)\big)\big]=\T\big[\Phi\big( \sum_{k\geq 1} x_k\big)\big].
\]
This proves the $\Phi$-moment case. The case of symmetric space  is obtain in a similar fashion by using the inequality
\[
\big\| U^+ I \Pi (\xi)\big\|_{F^{(2)}(\M\overline{\otimes}B(\ell_2(\mathbb{N}^2)))}\lesssim_F \big\| \xi\big\|_{F^{(2)}(\M\overline{\otimes}B(\ell_2(\mathbb{N})))}
\]
where $F^{(2)}$ is the $2$-convexfication of $F$.  Alternatively,  we may also deduce
\eqref{DD-sym} from the $\Phi$-moment case by using \cite[Theorem~7.1]{KaltonSMS}. 
 \end{proof}
 \begin{remark}
 One can also deduce Proposition~\ref{DD} from the  recent result on distribution form of the dual Doob inequality  in \cite{Jiao-Sukochev-Wu-Zanin}.
 \end{remark}

 The next result is an improvement  of \cite[Theorem~4.7]{Ran-Int}. The argument outlined   below is much simpler than the one from \cite{Ran-Int}.

\begin{theorem} Let $0<p< 2$.
If  $F\in {\rm Int}(L_p, L_2)$  then there exists a constant $C_F$ such that for every $x \in \h_F^c(\M)$, the following two inequalities hold:
\[
\big\|x\big\|_{\H_F^c(\M)} \leq C_F \big\|x\big\|_{\h_F^c(\M)}
\]
and
\[
\big\|x\big\|_{F(\M)} \leq C_F \big\|x\big\|_{\h_F^c(\M)}.
\]
Similarly, if $\Phi$ is $p$-convex and $2$-concave for $0<p<2$, then  there exists a constant $C_{p}$ so that for every $x\in \h_\Phi^c(\M)$, we have
\[
\max\Big\{ \T\big[ \Phi\big(S_c(x)\big)\big] ;  \T\big[ \Phi\big(|x|\big)\big]  \Big\} \leq C_{p} 
\T\big[ \Phi\big(s_c(x)\big)\big].
\]
\end{theorem}
\begin{proof}
We use \cite[Theorem~4.11]{Jiao-Ran-Wu-Zhou} which states that for $0<r\leq 2$ and $x \in \h_r^c(\M)$, then  $\|dx\|_{L_r(\M;\ell_2^c)} \leq \sqrt{2/r} \|x\|_{\h_r^c(\M)}
$
and
$
\|x\|_{r} \leq \sqrt{2/r}\|x\|_{\h_r^c(\M)}.
$
We can immediately  deduce the desired inequalities by interpolation. The proof for the $\Phi$-moments is identical to the second part of the proof of Theorem~\ref{reverse-cond} so we omit the details.
\end{proof}

The next result is an extension of the noncommutative  L\'epingle-Yor  inequality to symmetric spaces of measurable operators.  This should be compared with  the corresponding noncommutative Stein inequality  treated in \cite[Lemma~3.3]{Jiao2} and 
\cite[Theorem~3.2]{Bekjan-Chen}.
\begin{theorem}\label{Lepingle-Yor}
Let $E\in {\rm Int}(L_1, L_q)$ for some $1<q<\infty$. There exists a constant $C_E$ so that for every  adapted sequence $(\xi_n)_{n\geq 1} \in E(\M;\ell_2^c)$,  the following holds:
\[
\big\| \big( \sum_{n\geq 1} |\E_{n-1}(\xi_n)|^2 \big)^{1/2}\big\|_{E(\M)} \leq C_E 
\big\| \big( \sum_{n\geq 1} |\xi_n|^2 \big)^{1/2}\big\|_{E(\M)}.
\]

Similarly, assume that  $\Phi$ is a convex Orlicz function that is $q$-concave for some $1<q<\infty$.  For every adapted   sequence $(\xi_n)_{n\geq 1} \in L_\Phi(\M;\ell_2^c)$,  the following holds:
\[
\T\Big(\Phi \big(\big( \sum_{n\geq 1} |\E_{n-1}(\xi_n)|^2 \big)^{1/2}\big)\Big) \lesssim_q 
\T \Big( \Phi\big( \big(\sum_{n\geq 1} |\xi_n|^2 \big)^{1/2}\big)\Big).
\]
\end{theorem}
\begin{proof}  We recall that the first inequality holds for $E=L_p$ when $1\leq p<\infty$ (\cite{PX,Qiu1}). The case  of general  function space $E\in {\rm Int}(L_1, L_q)$  is a simple consequence of the fact that  the spaces of adapted sequences interpolate. For the second part,  let $\xi=(\xi_n)_{n\geq 1}$ be an adapted sequence in $L_\Phi(\M;\ell_2)$. Denote by $T(\xi)$ the sequence $(\E_{n-1}(\xi_n))_{n\geq 1}$. One can easily check (using the first part) that  for every $t>0$, the following holds:
\[
K( T(\xi), t, L_1(\M;\ell_2^c), L_q(\M;\ell_2^c)) \lesssim_{q} K( \xi, t, L_1^{\rm ad}(\M;\ell_2^c), L_q^{\rm ad}(\M;\ell_2^c)). 
\]
It follows from  $K$-closedness that for every  $t>0$,
\[
K( \cal{S}_c(T(\xi)), t ; L_1(\M), L_q(\M)) \lesssim_{q} K( \cal{S}_c(\xi), t ;  L_1(\M), L_q(\M)). 
\]
We can deduce as in the last part of the proof of Theorem~\ref{reverse-cond} that
\[
\int_0^\infty \Phi(\mu_t(\cal{S}_c(T(\xi))))\ dt  \lesssim_q \int_0^\infty \Phi(\mu_t(\cal{S}_c(\xi)))\ dt
\]
which is precisely the desired inequality.
\end{proof}

We conclude  this section with some   discussions on  Davis type inequalities associated with symmetric spaces of operators initiated in \cite{Ran-Wu-Xu}. Below, for a symmetric Banach function space $E$, the notation 
$ E^{\rm cond, ad}(\M;\ell_2^c)$ is used for  the subspace of $E^{\rm cond}(\M;\ell_2^c)$ consisting of adapted sequences.

\begin{proposition}\label{Davis-symmetric}
Let $E$ be a  symmetric Banach function space. 
\begin{enumerate}[{\rm(i)}]
\item If  $E\in {\rm Int}(L_1,L_q)$ for some $1<q<\infty$, then the following inclusion holds:
\[
E^{\rm ad}(\M;\ell_2^c)  \subseteq  E(\oplus_{n=1}^\infty \M_n) +  E^{\rm cond, ad}(\M;\ell_2^c).
\]
\item  If  $E\in {\rm Int}(L_1,L_2)$,  then following identity holds:
\[
E^{\rm ad}(\M;\ell_2^c)  = E(\oplus_{n=1}^\infty \M_n) +  E^{\rm cond, ad}(\M;\ell_2^c).
\]
\end{enumerate}
\end{proposition}
We only indicate the adjustments in the  argument as it follows verbatim  the one used in \cite{Ran-Int}. First, we observe that the simultaneous decomposition in \cite[Proposition~4.9]{Ran-Int}  actually applies to all $2/3<p<p_0$  for any given $2<p_0<\infty$. Then the proof of \cite[Proposition~4.11]{Ran-Int}  can be carried out  with  arbitrary $1<q<\infty$ since we no longer have any restriction on the $K$-closedness. This  gives the  inclusion in the first item. 

On the other hand, 
if   $E\in {\rm Int}(L_1,L_2)$,  one can easily see that $E(\oplus_{n=1}^\infty \M_n) \subseteq E^{\rm ad}(\M; \ell_2^c)$ and Theorem~\ref{reverse-cond} gives $E^{\rm cond, ad}(\M;\ell_2^c) \subseteq  E^{\rm ad}(\M; \ell_2^c)$. The equality then follows from combining these facts with the first item. \qed

\medskip

As an immediate application of Proposition~\ref{Davis-symmetric}, we can answer   problems from \cite[Remark~3.11]{Ran-Wu-Xu} and \cite[Problem~4.2]{RW}.
\begin{corollary}\label{Davis}
Let $E$ be a symmetric Banach function space. If $E\in {\rm Int}(L_1, L_2)$, then the following identity holds:
\[
\H_E^c(\M) =\h_E^d(\M) +\h_E^c(\M).
\]
Consequently, the mixed Hardy spaces also coincide: $\H_E(\M)=\h_E(\M)$.
\end{corollary}

Clearly, the first identity follows from Proposition~\ref{Davis-symmetric}(ii) while the second can easily be deduced from combining the first one with its row version.

We remark that Corollary~\ref{Davis}  is sharp in the sense that  if  the conclusion is valid then $E\in {\rm Int}(L_1, L_2)$. This can seen as follows: if $\H_E^c(\M) =\h_E^d(\M) +\h_E^c(\M)$ then a fortiori,  $\h_E^c(\M)\subseteq \H_E^c(\M)$. Therefore, there exists a constant $C$ so that for any $x\in \h_E^c(\M)$, $\|x\|_{\H_E^c} \leq C\|x\|_{\h_E^c}$. Fix  a $\sigma$-field $\cal{F}$ whose atoms have finite measure. Denote by $\E(\cdot)$ the conditional expectation $\mathbb{E}(\cdot|\cal{F})$.  

Let $\M=L_\infty$. Fix $f \in E$ and consider the finite  martingale  defined  by setting $f_1=\E(f)$ and $f_2=f$. 
First, we note that since $E$ is a  symmetric Banach function  space, by the boundedness of conditional expectations on $E$, we have $\|\E(|f|)\|_E \leq \|f\|_E$. On the other hand,  we  make the following estimates:
\begin{align*}
\|f\|_E &\leq \| f-f_1\|_E +\|f_1\|_E\\
&\leq \| ( |f_1|^2 +|f-f_1|^2)^{1/2} \|_E + \|f_1\|_E\\
&\leq C\| ( |f_1|^2 + \E|f-f_1|^2)^{1/2} \|_E + \|f_1\|_E\\
&\leq C\| (  \E(|f|^2)^{1/2} \|_E + \|f_1\|_E.
\end{align*}
From the inequality  $|f_1|^2 \leq \E(|f|^2)$, we get that $\|f\|_E\leq (C+1)\| (  \E(|f|^2)^{1/2} \|_E$. 
We can now conclude  from 
 \cite[Theorem~7.2]{KaltonSMS} that $E \in {\rm Int}(L_1,L_2)$.


\def\cprime{$'$}
\providecommand{\bysame}{\leavevmode\hbox to3em{\hrulefill}\thinspace}
\providecommand{\MR}{\relax\ifhmode\unskip\space\fi MR }
\providecommand{\MRhref}[2]{%
  \href{http://www.ams.org/mathscinet-getitem?mr=#1}{#2}
}
\providecommand{\href}[2]{#2}

\end{document}